\documentclass[a4paper]{article} 

\usepackage{a4wide}

\usepackage[usenames,dvipsnames]{xcolor}
\usepackage[english]{babel} 
\usepackage[color=cyan]{todonotes}
\usepackage{float}
\usepackage{breakcites}
\newcommand{\R}{\mathbb R}
\usepackage{tcolorbox}
\tcbuselibrary{skins}
\usepackage{colortbl}
\usepackage{bbm}

\usepackage{mathtools}
\usepackage{tabularx}
\usepackage{array}
\usepackage{colortbl}
\usepackage{graphicx}
\usepackage{subfigure}
\usepackage{pgfplots}
\pgfplotsset{width=\linewidth,compat=newest}

\usepackage{nicefrac}

\usepackage[utf8]{inputenc}
\usepackage[T1]{fontenc}  

\definecolor{color1}{RGB}{0,139,0} 
\definecolor{color2}{RGB}{154,255,154} 

\numberwithin{equation}{section}

\newcommand{\Z}{\mathbb{Z}} 
\newcommand{\N}{\mathbb{N}} 
\newcommand{\ra}{\rightarrow}

\newcommand{\abs}[1]{\left |#1\right |}

\usepackage{xspace}
\newcommand{\EulerGamma}{\gamma}
\newcommand{\push}{\emph{push}\xspace}

\newcommand{\Push}{\emph{push}\xspace}

\usepackage{hyperref} 
\usepackage{url}
\usepackage[framemethod=tikz]{mdframed}
\usepackage{cuted}
\usepackage{amsmath}
\usepackage{enumitem}
\usepackage{datetime}
\usepackage{bm}

\hypersetup{hidelinks,colorlinks,breaklinks=true,urlcolor=color2,citecolor=color1,linkcolor=color1,bookmarksopen=false,pdftitle={Title},pdfauthor={Author}}
\usepackage{amssymb,amsmath,amsthm,amsfonts}
\makeatletter
\def\@endtheorem{\endtrivlist}
\makeatother

\newtheoremstyle{abcd}
  {}
  {}
  {\itshape}
  {}
  {\bfseries}
  {.}
  {.5em}
  {}

\theoremstyle{abcd}
\newtheorem{theorem}{Theorem}
\numberwithin{theorem}{section}
\newtheorem{definition}[theorem]{Definition}

\newtheorem{corollary}[theorem]{Corollary}

\newtheorem{lemma}[theorem]{Lemma}
\newtheorem{remark}[theorem]{Remark}

\DeclarePairedDelimiter{\ceil}{\lceil}{\rceil}
\DeclarePairedDelimiter{\floor}{\lfloor}{\rfloor}

\usepackage{tablefootnote}

\begin{document}

\title{Asymptotics for Push on the Complete Graph}

\author{Rami Daknama$^1$ \and Konstantinos Panagiotou$^1$ \and Simon Reisser$^1$}

\date{$^1$Ludwig-Maximilians-Universität München\\~\\24$\textnormal{th}$ March, 2020}
\maketitle 

\begin{abstract}
We study the popular randomized rumour spreading protocol \push. Initially, a node in a graph possesses some information, which is then spread in a round based manner. In each round, each informed node chooses uniformly at random one of its neighbours and passes the information to it. The central quantity to investigate is the \emph{runtime}, that is, the number of rounds needed until every node has received the information.

The \push protocol and variations of it have been studied extensively. Here we study the case where the underlying graph is complete with $n$ nodes. Even in this most basic setting, specifying the limiting distribution of the runtime as well as determining  related quantities, like its expectation, have remained open problems since the protocol was introduced. 

In our main result we describe the limiting distribution of the runtime. We show that it does not converge, and that it becomes, after the appropriate normalization, asymptotically periodic both on the $\log_2n$ as well as on the $\ln n$ scale. In particular, the limiting distribution converges only if we restrict ourselves to suitable subsequences of $\mathbb N$, where simultaneously $\log_2 n-\floor{\log_2n}\to x$ and $\ln n-\floor{\ln n}\to y$ for some fixed $x,y\in [0,1)$. On such subsequences we show that the expected runtime is $\log_2 n+\ln n+h(x,y)+o(1)$, where $h$ is explicitly given and numerically  $|\sup h - \inf h| \approx 2\cdot 10^{-4}$.

This ``double oscillatory'' behaviour has its origin in two key ingredients that were also implicit in previous works: first, an intricate discrete recursive relation that describes how the set of informed nodes grows, and second, a coupon collector problem with batches of size $n$ that takes the lead when the protocol is almost finished. Rounding in the recursion introduces the periodicity on the $\log_2 n$ scale -- as it is the case in many discrete systems -- and rounding in the batched problem is the source of the second periodicity.
\end{abstract}



\thispagestyle{empty} 

\section{Introduction}
We consider the well-known and well-studied rumour spreading protocol \textit{Push}. It has applications in replicated databases \cite{demers_epidemic_1987}, multicast \cite{birman1999bimodal} and blockchain technology \cite{8975840}. \textit{Push} operates on graphs and proceeds in rounds as follows. In the beginning, one node has a piece of information. In subsequent rounds each informed node chooses a neighbour independently and uniformly at random and informs it. For a graph $G=(V,E)$ with $|V|=n$ and a node $v\in V$  we denote by $X(G,v)$ the (random) number of rounds needed to inform all  nodes, where at the beginning of the first round only $v$ knows the information. We call $X(G,v)$ the \emph{runtime} (on $G$ with start node $v$). The most basic case, and the one that we study here, is when $G$ is the complete graph $K_n$. Since in that case the initially informed node makes no difference, we will abbreviate $X(K_n,v)=X_n$ for any starting node $v$. 

\paragraph{Related Work}
There are several works studying the runtime of \push on the complete graph. The first paper considering this protocol is by Frieze and Grimmett~\cite{frieze_the_1985}, who showed that with high probability (whp), that is,  with probability $1-o(1)$ as $n\to\infty$, that 
$$X_n = \log_2 n  + \ln n  + o(\ln n).$$
Moreover, they obtained bounds for (very) large deviations of $X_n$ from its expectation.
In \cite{pittel_on_1987}, Pittel improved upon the results in~\cite{frieze_the_1985}, in particular, he showed that for any $f:\N\ra \R^+$ with $f=\omega(1)$, whp,
$$|X_n - \log_2n-\ln n|\leq f(n).$$ 
The currently most precise result in this context was obtained by  Doerr and Künnemann~\cite{doerr_tight_2014}, who considered in great detail the distribution of $X_n$. They showed that $X_n$ can be stochastically bounded (from both sides) by coupon collector type problems. This gives a lot of control regarding the distribution of $X_n$, and it allowed them to derive, for example, very sharp bounds for tail probabilities. Apart from that, it enabled them to consider related quantities, as for example the expectation of $X_n$. Among other results, their bounds on the distribution of $X_n$ imply that  
\begin{align}
\label{push_precise_bound_for_runtime_2014}
\floor*{\log_2n} + \ln n - 1.116 \leq \mathbb{E}[X_n] \leq \ceil*{\log_2n} + \ln n + 2.765,
\end{align}
which pins down the expectation up to a constant additive term. Besides on complete graphs, \push has been extensively studied on several other graph classes. For example, Erdős-Rényi random graphs \cite{feige_randomized_1990, fountoulakis_reliable_2010}, random regular graphs and  expander graphs~\cite{fountoulakis_rumor_2010, Panagiotou2015, Daknama2019}. More general bounds that only depend on some graph parameter have also been derived, e.g.~the diameter  \cite{feige_randomized_1990}, graph conductance~\cite{mosk-aoyama_fast_2008, chierichetti_almost_2010,chierichetti_rumour_2010, giakkoupis_tight_2011} and node expansion \cite{chierichetti_rumour_2010, sauerwald_rumor_2011, giakkoupis_rumor_2012, giakkoupis_tight_2014}.

\paragraph{Results}

In order to state our main result we need some definitions first. Set
\begin{align*}
 g=g^{(1)}:[0,1] \ra [0,1],~~ x \mapsto xe^{x-1} \quad \text{and}\quad g^{(i)}:[0,1] \ra [0,1], ~~ g^{(i)}=g\circ g^{(i-1)}, i\ge 2.
\end{align*}
As we will see later, the function $g$ describes, for a wide range of the parameters, the evolution of the number of uninformed nodes; in particular, if at the beginning of some round there are $xn$ uninformed nodes, then at the end of the same round there will be (roughly) $g(x)n$ uninformed nodes, and after $i$ rounds there will be (roughly) $g^{(i)}(x)n$ uninformed nodes. This fact is not new -- at least for bounded $i$ -- and has been observed long ago, see for example \cite[Lem.~2]{pittel_on_1987}. For $x\in \mathbb{R}$ define the function 
\begin{equation}
\label{eq:c}
c(x)= -x + \lim\limits_{a \ra \infty, a \in \N} \lim\limits_{b \ra \infty, b \in \N} - a + b + \ln\big(g^{(b)}(1-2^{-a-x})\big),
\end{equation}
whose actual meaning will become clear later. We will show that the double limit exists, so that this indeed defines a function $c:\mathbb{R} \to \mathbb{R}$. Moreover, we will show that $c$ is continuous and periodic with period 1, that is,  if we write $\{x\}=x-\floor{x}$  then $c(x) = c(\{x\})$,  and that (numerically)  $|\sup c - \inf c| \approx 10^{-9}$, cf.~Figure~\ref{figure_c}.
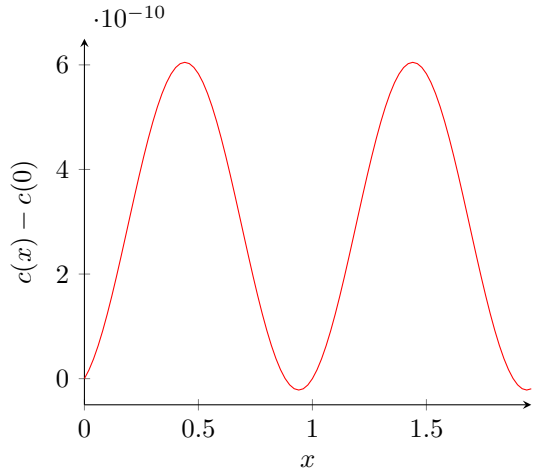
\begin{figure}[t]
    \centering
    \begin{minipage}{0.5\textwidth}
\centering
 		\begin{tikzpicture}
    		\begin{axis}[ymin=-0.00000000005,ymax=0.00000000065, axis x line=bottom, axis y line=left, xlabel=$x$,ylabel=$c(x)-c(0)$]
      			\addplot[no markers, color=red]  file {c.txt};
    		\end{axis}
  		\end{tikzpicture}
    \end{minipage}
		       \label{figure_c}
  		\caption{\small The function $c(x)-c(0)$, $c(0)\approx0.105$, plotted for values of $x$ between 0 and 2. The periodic nature of the 						function and its small amplitude are  evident.}
\end{figure}
The Gumbel distribution will play a prominent role in our considerations. 
We say that a real valued random variable $G$ follows a Gum$(\alpha)$ distribution with parameter $\alpha\in \mathbb{R}$,  $G \sim \text{Gum}(\alpha)$, if for all $x \in \R$ 
 $$P[G\leq x]=e^{-e^{-x-\alpha}}, \quad x\in\mathbb{R}.$$
Finally, let $\EulerGamma$ denote the Euler-Mascheroni constant. With all these ingredients we can now state our main result, which specifies -- see also below -- the distribution of the runtime of \push on the complete graph.
\begin{lemma}
\label{main_theorem_push_complete}
Let $G \sim \textnormal{Gum}(\gamma)$. Then, as $n\to\infty$
\begin{align*}
\sup_{k\in \N}\Big |P[X_n \geq k]-P\Big[\big\lceil G + \log_2n +\ln n+\gamma+c(\{\log_2 n\})\big\rceil \geq k\Big]\Big | =o(1). 
\end{align*}
\end{lemma}
This lemma does not look completely innocent, and it actually has striking consequences. It readily implies the following result, which establishes that the limiting distribution $X_n$ is periodic both on the $\log_2n$ and on the $\ln n$ scale. In order to formulate it, we need a version of the Gumbel distribution where we restrict ourselves to integers only. More specifically, we say that a random variable $G$ follows a \emph{discrete Gumbel} distribution, $G\sim$ dGum$(\alpha)$, if the domain of $G$ is $\mathbb{Z}$ and
 $$P[G\leq k]=e^{-e^{-k-\alpha}}, \quad  k \in \mathbb{Z}.$$
\begin{theorem}
\label{main_theorem_corollary}
Let $x,y\in [0,1)$ and $(n_i)_{i\in \mathbb{N}}$ be a strictly increasing sequence of natural numbers, such that $\log_2n_i-\floor{\log_2 n_i}\to x $ and  $\ln n_i-\floor{\ln n_i}\to y$ as $i\to \infty$. Then in distribution, as $i\to\infty$
\begin{align*}
X_{n_i} - \big(\floor{\log_2 n_i}+\floor{\ln n_i}\big)\ra \text{\normalfont dGum}(-x-y-c(x)).
\end{align*}
\end{theorem}
Some remarks are in place. First, it is a priori not obvious (at least it was not to us) that subsequences as required in the theorem indeed exist. They do, and the fundamental reason for this is that real numbers can be approximated arbitrarily well by rational numbers; we include a short proof of the existence in the Appendix. Second, it is a priori not clear that $x + c(x)$ is not constant for $x\in[0,1)$. If it was constant,  Theorem~\ref{main_theorem_corollary} would imply that the limiting distribution of $X_n$ is periodic on the $\ln n$ scale only. Although we didn't manage to \emph{prove} that $x + c(x)$ is not constant, we have stong numerical evidence that it indeed is not so. In particular, as we shall also see later, the double limit in the definition of $c$ converges exponentially fast and thus it is not difficult to obtain accurate estimates for it and explicit error bounds. We leave it as an open problem to study the behavior of $c$ more accurately.

Our next result addresses moments of $X_n$. Bounds given in previous works, for example in~\cite{doerr_tight_2014}, guarantee that $X_n - \log_2 n- \ln n$ and all integer powers of it are uniformly integrable. This allows us to conclude that the expectation and all of its moments also converge. 
\begin{theorem}
\label{main_theorem_expectation}
Let $x,y\in [0,1)$ and $(n_i)_{i\in \mathbb{N}}$ be a strictly increasing sequence of natural numbers, such that $\log_2n_i-\floor{\log_2 n_i}\to x $ and  $\ln n_i-\floor{\ln n_i}\to y$ as $i\to \infty$. Then for all $k\in \N$, as $i\to\infty$
\begin{align*}
\mathbb{E}\Big[\big(X_{n_i} - (\floor{\log_2 n_i}+\floor{\ln n_i})\big)^k\Big] \rightarrow \mathbb{E}\Big[\big(\text{\normalfont dGum}(-x-y-c(x))\big)^k\Big]. 
\end{align*}
\end{theorem}
For $x,y \in [0,1)$ and a strictly increasing sequence of natural numbers $(n_i)_{i\in \mathbb{N}}$ such that $\{\log_2 n_i\} \to x $ and  $\{\ln n_i\}\to y$ Theorem \ref{main_theorem_expectation} immediately implies that, as $i\to\infty$,
\begin{equation*}
\label{eq:Eh}
	\mathbb{E}\big[X_{n_i}\big] = \log_2n_i + \ln n_i + h(x,y) + o(1),
\end{equation*}
where we abbreviated $h(x,y) = \mathbb{E}\big[\text{dGum}(-x-y-c(x))\big] -x-y$, cf.~Figure~\ref{figure_h} for a visualization of $h$. Similarly, to obtain an expression for the variance of the runtime, see that 
\begin{align*}
\textnormal{Var}[X_{n_i}]&=\textnormal{Var}\big[X_{n_i} - (\floor{\log_2 n_i}+\floor{\ln n_i})\big]\\
&=\mathbb{E}\Big[\big(X_{n_i} - (\floor{\log_2 n_i}+\floor{\ln n_i})\big)^2\Big]-\mathbb{E}\big[X_{n_i} - (\floor{\log_2 n_i}+\floor{\ln n_i})\big]^2
\end{align*}
and using Theorem \ref{main_theorem_expectation}, consequently  
\begin{align*}
\textnormal{Var}[X_{n_i}]=\mathbb{E}\big[\text{dGum}(-x-y-c(x))^2\big]-\mathbb{E}\big[\text{dGum}(-x-y-c(x))\big]^2 + o(1).
\end{align*}
To determine the expectation and variance of the runtime we need to consider various moments of the discrete Gumbel distribution. To this end, let $X$ be an integer valued random variable with finite $kth$ moment, then
\begin{align*}
\mathbb{E}\big[X^k\big]=\sum_{\ell\in\Z}\ell^kP[X=\ell]=\sum_{\ell\in\Z}\ell^k\big(P[X\le\ell]-P[X\le \ell-1]\big),
\end{align*}
and therefore, for all $\alpha\in\R$ and $k\in\mathbb{N}$,
\begin{align*}
\mathbb{E}\big[\text{dGum}(\alpha)^k\big]=\sum_{\ell\in\Z}\ell^k\left (e^{-e^{-\ell-\alpha}} - e^{-e^{-\ell-\alpha+1}}\right ).
\end{align*}
This sum converges exponentially fast, both for $\ell\to \infty$ and $\ell\to-\infty$, and thus allows for effective numerical treatment. In summary, improving \eqref{push_precise_bound_for_runtime_2014}, we get for all $n\in \mathbb{N}$ the numerical bounds
\begin{align*}
\log_2 n+\ln n+1.18242 \le \mathbb{E}[X_n]\le \log_2 n+\ln n+1.18263,
\end{align*}
as $\inf_{0 \le x,y\le 1} h(x,y)= 1.18242 \dots$, $\sup_{0 \le x,y\le 1} h(x,y)= 1.18262\dots$
and 
\begin{align*}
1.7277 \le \textnormal{Var}[X_n]\le 1.7289.
\end{align*}
These numerical bounds are (essentially) best possible, see also Figure \ref{figure_h}. Higher moments can be estimated similarly.

\begin{figure}[t]
    \centering
    \begin{minipage}[t]{0.47\textwidth}
    	    	\centering
\begin{tikzpicture}
\begin{axis}[3d box, view={-225}{45}, xlabel=$x$,ylabel=$y$,zlabel=${h(x,y)}-1.182$, z label style={at={(axis description cs:-0.26,0.5)},anchor=north}, zmax=0.00065]
     		\addplot3[surf,mesh/rows=30,shader=faceted interp, opacity=0.9] file {h.txt};
   		 \end{axis}
  \end{tikzpicture}
  \end{minipage}\hfill
  \begin{minipage}[t]{0.47\textwidth}
  	\centering
 	\begin{tikzpicture}
\begin{axis}[3d box, view={-225}{45}, xlabel=$x$,ylabel=$y$,zlabel=${\text{Var}[X_{n_i}]-1.72}$, z label style={at={(axis description cs:-0.26,0.5)},anchor=north}, zmax=0.009]
     		 \addplot3[surf,mesh/rows=30,shader=faceted interp, opacity=0.9] file {Var.txt};
   		 \end{axis}
  \end{tikzpicture}
 \end{minipage}
 \caption{\small Let $(n_i)_{i\in \mathbb{N}}$ be a sequence of natural numbers such that $\{\log_2 n_i\} \to x $ and  $\{\ln n_i\}\to y$ for $x,y\in[0,1)$. The left figure shows the function $h(x,y)$ (appearing in the expectation of $X_{n_i}$) for values of $x$ and $y$ between 0 and 1. The right figure  shows $\text{Var}[X_{n_i}]$ as a function of $x,y$.}
   \label{figure_h}
\end{figure}
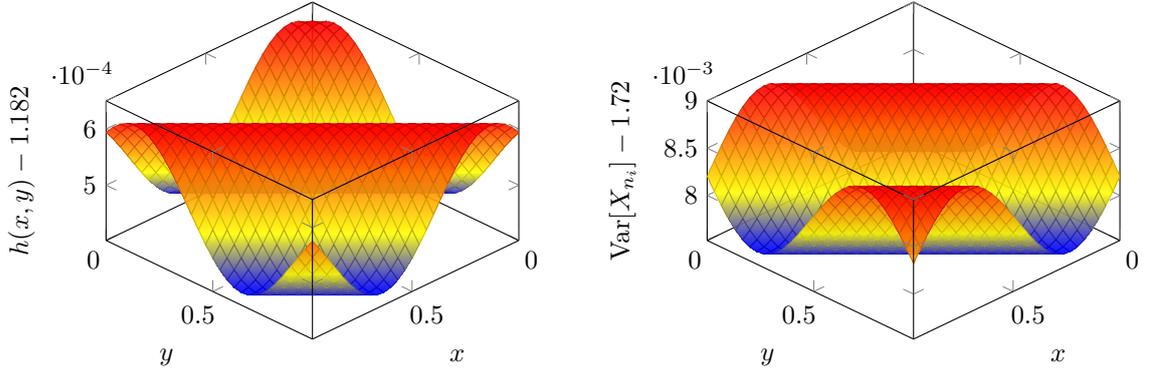
Let us close this section with a final remark on the function $c$ defined in \eqref{eq:c}, as this might be helpful in future works. This function is defined as the limit of a sequence in two parameters $a,b$; the main reason for this is its combinatorial origin, which will become apparent in the proofs. However, all that is actually important is that $b$ is large enough, in the sense that the difference $b-a\to \infty$ as $a\to\infty$. So, if we write $h$ for an integer function that diverges to infinity, then we could define
\[
d(x)= -x + \lim\limits_{a \ra \infty, a \in \N} h(a) + \ln\big(g^{(a + h(a))}(1-2^{-a-x})\big).
\]
Then $c(x) = d(x)$  (which we state without proof, as we do not need it here), and $c$ can be represented as a limit of an (one-dimesional) sequence.
\paragraph{Outline} 
In the next section we give an outline of the proof of our main results. At the beginning of the rumour spreading process \push is dominated by an exponential growth of the informed nodes (Lemma \ref{lemma1}). For the main part, where most nodes get informed, it closely follows a deterministic recursion (Lemma \ref{lemma2}) and at the end it is described by a coupon collector type problem (Lemma \ref{lemma3}). Based on these lemmas we give the rigorous proof of our claims in Section \ref{proof_main_sec}. The proof to these three important lemmas can also be found there, in Subsections \ref{proof_lemma2}-\ref{proof_lemma3}. Subsections \ref{proofs_of_the_preparation_results_push_exact_a} and \ref{proofs_of_the_preparation_results_push_exact} contain all other proofs.

\paragraph{Further Notation}
Unless stated otherwise, all asymptotic behaviour in this paper is for $n\ra \infty$. Consider a graph $G=(V,E)$. For $t \in \N_0 ~(=\N\cup \{0\})$ we denote by $I_t\subseteq V$ the set of informed nodes at the end of round $t$; in particular $|I_0|=1$. Analogously we write $U_t=V\backslash I_t$ for the set of uninformed nodes.  
For an event $A$, we sometimes write $P_A[\cdot] $ instead of $P[\cdot \mid A]$ to denote the conditional probability and we write $\mathbb{E}_A[\cdot]=\mathbb{E}[\cdot \mid A]$. If we condition on $I_t$, then we also abbreviate $P[\cdot \mid I_t]=P_t[\cdot]$ and $\mathbb{E}[\cdot \mid I_t]=\mathbb{E}_t[\cdot]$. 



\section{Proof Overview}
\label{proof_overview}

Let us start the proof of Lemma~\ref{main_theorem_push_complete} about the distribution of the runtime of \push on $K_n$ with a simple observation, that is more or less explicit also in previous works. Note that as long as the \emph{total number} of pushes performed is $o(\sqrt{n})$, then whp no node will be informed twice -- this is a simple consequence of the famous birthday paradox. That is, whp as long as $|I_t| = o(\sqrt{n})$, every node in $I_t$ will inform a currently uninformed node and thus $|I_{t+1}| = 2|I_t|$. In particular, whp
\begin{align}\label{eq:doubling}
	|I_{t_0}| = 2^{t_0}, ~~ \text{where} ~~ t_0 := \floor*{0.49\cdot  \log_2n}.
\end{align}
Soon after round $t_0$ things get more complicated. We continue with a definition. Apart from the functions $g^{(i)}$ defined in the previous section, we will also need the following functions. Set
\begin{align*}
 f=f^{(1)}:[0,1] \ra [0,1],~~x \mapsto 1-e^{-x}(1-x)
~~\text{and}~~
f^{(i)}:[0,1] \ra [0,1],~~
f^{(i)}=f\circ f^{(i-1)}, i\ge 2.
\end{align*}
Some elementary properties of $f$ are: $f$ is strictly increasing and concave, and $f^{(b)}(x) \to 1$ as $b \to \infty$ for all $x \in (0,1]$. Moreover, $f^{(i)}(x)=1-g^{(i)}(1-x)$ for all $x\in[0,1]$ and $i\in\mathbb{N}$.  It is also not difficult to establish, see also~\cite{pittel_on_1987} and Lemma \ref{expboundapplied} below, that $f$ captures the behavior of the expected number of informed nodes after one round of the protocol. Moreover, $|I_{t+1}|$ is typically close to $f(|I_t|/n)n$. Here we will need a more explicit qualitative control of how $|I_t|$ behaves, since our aim is to specify the limiting distribution. We show the following statement, which implies that if we start in round $t_0$ (set $T = t_0$ in that lemma) then whp for \emph{all} succeeding rounds $t_0 + t$ the number of informed nodes is close to $f^{(t)}(|I_{t_0}/n|)n$.
\begin{lemma}
\label{lemma2}
\label{6}
Let $0<c < 0.49$ and $T \ge c\log_2n$. Then
$$P_T\left [\bigcap_{t\in \N_0} \left \{\big ||I_{T+t}|-f^{(t)}\left (|I_{T}|/n\right )n\big |\le n^{1-{c}/4}\right \}\right ]=1-O(n^{-c^2/10}).$$
\end{lemma}
Thus, the key to understanding $|I_t|$ is to understand how $f$ behaves when iterated very many times. Note that when the number of informed nodes is $xn$ for some very small $x$, then the $e^{-x}$ term in the definition of $f$ can be approximated by $1-x$ and therefore $f(x)\approx 1-(1-x)^2\approx 2x.$ This crude estimate suggests that the number of informed nodes doubles every round as long as there are only few informed nodes, and we know already that the doubling is perfect if $xn = o(\sqrt{n})$. Our next lemma actually shows that the doubling continues to be \emph{almost} perfect, as long as the total number of nodes is not close to $n$.
\begin{lemma}
\label{lemma1}
\label{7}
Let $a,T\in \N$ be such that $2^{-a}<0.1$ and $T\le\floor*{ 0.49\cdot  \log_2n}$. Set $t_1:= \floor*{\log_2n}-a$. Then 
\begin{equation*}
\left |2^{t_1}-f^{(t_1-T)}\left (2^{T}/n\right )n\right |\le 2^{-2a+1}n.
\end{equation*}
\end{lemma}
Combining the previous lemmas we have thus established that for any $a\in \N$ with $2^{-a}<0.1$ whp
\begin{equation}
\label{eq:ft1t0}
	(1 - 2^{-a+2})\cdot 2^{t_1}\le |I_{t_1}| \le 2^{t_1},
	\quad t_1 := \floor*{\log_2n}-a.
\end{equation}
Here we can think of $a$ being very large (but fixed) and then the two bounds are very close to each other; in particular, $|I_{t_1}| \approx 2^{\floor*{\log_2n}-a}$ and thus $I_t$ contains a linear number of nodes. Up to that point we have studied the behaviour of the process up to time $t_1$. Next we perform another~$b$ steps, where again $b$ is fixed. Applying Lemma \ref{lemma2} once more  and using that $f^{(b)}(x)$ is increasing and is less than $ 1$ for $x<1$ yields with room to spare that whp
\begin{align}\label{eq:start_main}
\left (1-n^{-\nicefrac{1}{6}}\right )f^{(b)}\big((1-2^{-a+2})2^{t_1} /n\big)
\leq  n^{-1}\left |I_{t_2}\right | 
\leq \left (1+ n^{-\nicefrac{1}{6}}\right )f^{(b)}\big(2^{t_1}/n\big),
~~t_2 := t_1+b.
\end{align}
In essence, this says that if we write $x = \log_2 n - \floor*{\log_2n} = \{\log_2 n\}$, then (we begin getting informal and obtain that)
\[
	|I_{t_2}| \approx  f^{(b)}\big(2^{t_1}/n\big)n = f^{(b)}\big(2^{-a-x}\big)n,
	\quad \text{where}\quad
	t_2 = \floor*{\log_2n}-a+b.
\]
In particular, choosing a priori $b$ large enough makes the fraction $|I_{t_2}|/n$ arbitrarily close to 1, that is, almost all nodes except for a tiny fraction are informed. All in all, up to time $t_2$ we have very fine control of the number of informed nodes, and we also see how the quantity $\{\log_2 n\}$ slowly sneaks in.

After time $t_2$ the behavior changes once more. In this regime there is an interesting connection to the well-known Coupon Collector Problem (CCP), which was also exploited in~\cite{doerr_tight_2014}. In order to formulate the connection, note that the number of pushes that are needed to inform one uninformed node, having $N$ informed nodes, is (in distribution) equal to the number of coupons needed to draw the $(N+1)$st distinct coupon. The CCP is very well understood, and it is a classic result that, appropriately normalized, the number of coupons tends to a Gumbel distribution. However, translating the number of required pushes to the number of rounds -- the quantity we are interested in -- is not straightforward. In particular, the number of pushes in one round depends on the current number of informed nodes. On the other hand, after round $t_2$ there are $n - o(n)$ informed nodes, so that we may hope to approximate the remaining number of rounds with $n^{-1}$ times  the number of coupons in the CCP. The next lemma establishes the precise bridge between the two problems. There, for two sequences of random variables $(X_n)_{n\in\mathbb{N}}$ and $(Y_n)_{n\in\mathbb{N}}$ we write $X_n\precsim Y_n$ if there is a function $h:\N\ra \R^+$ with $h =o(1)$ such that $P[X_n\geq x]  \leq P[Y_n \geq x]+ h(n)$  for all $n \in \N,x \in \R$; $X_n\succsim Y_n$ is defined with ``$\ge$'' instead of ``$\le$''.
\begin{lemma}
\label{lemma3}
Let $G\sim \textnormal{Gum}(\gamma), b>2a\in \mathbb{N}$ and assume that $\ell \cdot n\le |I_{\floor{\log_2 n}-a+b}|\le u\cdot n$ for some $\ell,u\in [0,1)$. Then 
\begin{align*}
X_n-\floor{\log_2 n}+a-b &\succsim  \left\lceil\ln n + \ln\left(\frac1u-1\right) + \gamma + \right\rceil
\end{align*}
and 
\begin{align*}
X_n-\floor{\log_2 n}+a-b &\precsim
 \left\lceil\ln n +  \ln\left(\frac1\ell-1\right) + \ln\left(\frac{\ell}{e\ell-e+1}\right) + \gamma + G\right\rceil.
\end{align*}
\end{lemma}
Note that the previous discussion guarantees that $\ell,u$ in Lemma  \ref{lemma3} are very close to 1 and very close to each other. So, the term $\ln({\ell}/({e\ell-e+1}))$  is very close to 0. We obtain that in distribution
\[
	X_n-\floor{\log_2 n}+a-b
	\approx \left\lceil\ln n +  \ln\left(\frac1u-1\right) + \gamma + G\right\rceil,
	\quad \text{where} \quad  u = f^{(b)}\big(2^{-a-x}\big).
\]
and equivalently with $x = \log_2 n - \floor*{\log_2n}$
\begin{equation}
\label{eq:completeboundX}
	X_n
	\approx \left\lceil\log_2 n+\ln n  -a + b + \ln\left(g^{(b)}(2^{-a-x})\right) -x + \gamma + G\right\rceil.
\end{equation}
Here we now encounter the mysterious function $c$ from~\eqref{eq:c}. The next lemma collects some important properties of it that will turn out to be very helpful.
\begin{lemma}
\label{conv}
\label{c_is_continuous_push_complete}
\label{converging}
The function $$c(x) = \lim\limits_{a \ra \infty, a \in \N} \lim\limits_{b \ra \infty, b \in \N} -a + b + \ln\big(g^{(b)}(1-2^{-a-x})\big)-x$$ is well-defined, continuous and periodic with period 1.
\end{lemma}
With all these facts at hand, the proof of Lemma~\ref{main_theorem_push_complete} is completed by considering the random variable on the right-hand side of~\eqref{eq:completeboundX}; in particular, the dependence on $y = \ln n - \floor*{\ln n}$ arises naturally. The complete details of the proof, which is based on Lemmas~\ref{lemma2}--\ref{lemma3} and follows the strategy outlined here can be found in Section~\ref{proof_main_sec} (together with the proofs of the lemmas). 

As described in the introduction, apart from the limiting distribution we are interested in the asymptotic expectation of the runtime. A key ingredient towards the proof of Theorem~\ref{main_theorem_expectation} is uniform integrability, which can be shown by using the distributional bounds from~\cite{doerr_tight_2014}. Uniform integrability is a sufficient condition that convergence in distribution also implies convergence of the means.
\begin{lemma}[uniform integrability]\label{uniform_integrability}
Let $k\in\N$ and set $Y_n:=X_n-\floor{\log_2 n}-\floor{\ln n}$. Then  $Y_n^k$ is uniformly integrable, that is 
$$\lim_{N\to\infty}\sup_{n\in \mathbb{N}}\mathbb{E}\Big[|Y_n|^k~\Big|~\mathbbm{1}\big[|Y_n|^k>N\big]\Big]= 0.$$
\end{lemma}
%

\section{Proof of the Main Result}\label{proof_main_sec}

In this section we complete the proof of Lemma \ref{main_theorem_push_complete} outlined in Section \ref{proof_overview}. Afterwards we give the (short) proofs for Theorems \ref{main_theorem_corollary} and \ref{main_theorem_expectation}.
\subsection{Proof of Lemma \ref{main_theorem_push_complete}}
As the outline was indeed rigorous until  \eqref{eq:start_main} we take the proof up from there. Choose the quantities $a,b\in\mathbb{N}$  such that $2a<b$ and recall that $t_1 =\floor{\log_2n}-a$. Set furthermore for brevity
\begin{align*}
\ell= \left (1-n^{-\nicefrac{1}{6}}\right )f^{(b)}\big((1-2^{-a+2})2^{t_1} /n\big)\quad\text{and}\quad 
u= \left (1+ n^{-\nicefrac{1}{6}}\right )f^{(b)}\big(2^{t_1}/n\big).
\end{align*}
Then \eqref{eq:start_main} states that, for $t_2=\floor{\log_2n}-a+b$,
\begin{align*}
\ell\leq n^{-1}\left |I_{t_2}\right | \leq u,
\end{align*}
and Lemma \ref{lemma3} yields, for $Y_n=X_n-\floor{\log_2n}+a-b$, that
\begin{align*}
Y_n&\precsim
\ceil*{\ln n + \ln\left(\frac1\ell-1\right) + \ln\left(\frac{\ell}{e\ell-e+1}\right)+ \gamma + G}
\end{align*}
and 
\begin{align*}
Y_n &\succsim  \ceil*{\ln n +\ln\left(\frac1u-1\right) +  \gamma + G}.
\end{align*}
The next lemma establishes that both $\ell,u$ tend to 1 as $a$ gets large, and moreover that the difference $\ln \left(1/\ell-1\right) - \ln \left(1/u-1\right)$ can be made arbitrarily small. Its proof can be found in Subsection \ref{proofs_of_the_preparation_results_push_exact}.
\begin{lemma}\label{u_l_to_1}
\label{lemma_u_l}
\label{u2_l2_close_enough}
\label{considering_u2_suffices}
For $\ell,u$ defined as above, where $b > 2a $
\begin{align*}
\lim_{a\to\infty}\sup_{n\in\N}|\ln\ell|=\lim_{a\to\infty}\sup_{n\in\N}|\ln u|=\lim_{a\to\infty}\sup_{n\in\N}\left |\ln\left(\frac{\ell}{e\ell-e+1}\right)\right |=0.
\end{align*}
Furthermore, 
\begin{align*}
\lim_{a\to\infty}\sup_{n\in\N}|\ln(1-\ell)-\ln(1-u)|=0.
\end{align*}
\end{lemma}
Thus, as $n\to \infty$,
$$\ln (1-u)=\ln\left (1-f^{(b)}\left ({2^{t_1}}/n\right )\right )+o(1)=\ln\left (g^{(b)}\left (1-2^{-a - \{\log_2n\}}\right )\right )+o(1).$$
Let $\varepsilon > 0$. Lemma \ref{u_l_to_1} readily implies that there are $a_0,n_0\in\N$ such that for all $a>a_0$ and $n>n_0$,  
\begin{align*}
&Y_n\succsim\ceil*{\ln n+\ln\left (g^{(b)}\left (1-2^{-a-\{\log_2 n\}}\right )\right )+\gamma + G -\varepsilon}
\end{align*}
and similarly also
\begin{align*}
&Y_n\precsim \ceil*{\ln n+\ln\left (g^{(b)}\left (1-2^{-a-\{\log_2 n\}}\right )\right )+\gamma + G + \varepsilon}.
\end{align*}
Lemma \ref{conv} guarantees that there is an $a_1 \ge a_0$  such that for all $a \ge a_1$
$$ \left| \ln\left (g^{(b)}\left (1-2^{-a-\{\log_2 n\}}\right )\right )-a+b - (c(\{\log_2n\})+\{\log_2n\})\right| \le \varepsilon.$$
Thus   for all $a>a_1$ and $n>n_0$
\begin{align*}
&X_n \succsim \ceil{\log_2n+\ln n + c(\{\log_2n\}) + \gamma + G - 2\varepsilon},
\end{align*}
as well as
\begin{align*}
&X_n \precsim \ceil{\log_2n+\ln n + c(\{\log_2n\}) + \gamma + G + 2\varepsilon}.
\end{align*}
Thus we are left with getting rid of the $\varepsilon$ terms in the previous equations. The following lemma accomplishes exactly that and therefore implies the claim of Lemma \ref{main_theorem_push_complete}. Its proof can be found in Subsection \ref{proofs_of_the_preparation_results_push_exact}.
\begin{lemma}\label{finishGumbel}
Let $h:\N \to\mathbb{R}^+$ and $G\sim\textnormal{Gum}(\gamma)$. Then 
\begin{align*}
\forall \varepsilon > 0: X_n\precsim \ceil{h(n)+G+\varepsilon} \implies
X_n\precsim \ceil{h(n)+G}.
\end{align*}
The respective statement also holds for ``$\succsim$''.
\end{lemma}
\subsection{Proof of Theorems \ref{main_theorem_corollary} and \ref{main_theorem_expectation}}
\begin{proof}[Proof of Theorem \ref{main_theorem_corollary}]
Recall that $\{z\}=z-\floor{z}, z\in\R$. Let $(n_i)_{i\in \mathbb{N}}$ be a strictly increasing subsequence of $\mathbb{N}$ such that $\{\log_2 n_i\}\to x$ and $\{\ln n_i\}\to y$. Substituting $k=\floor{\log_2 n_i}+\floor{\ln n_i}+1+t$ for any $t\in\mathbb{Z}$  we get that
\begin{align*}
~&~P\Big[\big\lceil G + \log_2n_i +\ln n_i+\gamma+c(\{\log_2 n_i\})\big\rceil \geq k\Big]\\
= &~P\Big[\big\lceil G + \log_2n_i +\ln n_i+\gamma+c(\{\log_2 n_i\})\big\rceil \geq \floor{\log_2 n_i}+\floor{\ln n_i}+1+t\Big]\\
= &~P\Big[\big\lceil G + \{\log_2n_i\} +\{\ln n_i\}+\gamma+c(\{\log_2 n_i\})\big\rceil >t\Big] \\
= &~P\Big[G + \{\log_2n_i\} +\{\ln n_i\}+\gamma+c(\{\log_2 n_i\}) >t\Big].
\end{align*}
Thus using Lemma \ref{main_theorem_push_complete}, Lemma \ref{conv} and Lemma \ref{finishGumbel} we get that, as $i\to\infty$,
\begin{align*}
\sup_{t\in \Z}\Big|P\big[X_{n_i} \geq \floor{\log_2 n_i}+\floor{\ln n_i}+1+t\big]-P\big[ G + x+y+\gamma+c(x) > t\big]\Big| =o(1). 
\end{align*}
Using the distribution function of $G\sim \text{Gum}(\gamma)$ we get 
\begin{align*}
P\big[X_{n_i} \geq \floor{\log_2 n_i}+\floor{\ln n_i}+1+t\big]\overset{i\to \infty}{\longrightarrow}  1-\exp\big (-\exp\left (-t+x+y+c(x)\right )\big ),
\end{align*}
that is,
\begin{align*}
P\big[X_{n_i} \leq \floor{\log_2 n_i}+\floor{\ln n_i}+t\big]\overset{i\to \infty}{\longrightarrow}  P(\text{\normalfont dGum}(-x-y-c(x)) \le t).
\end{align*}
\end{proof}
Next we prove Theorem \ref{main_theorem_expectation}.
\begin{proof}[Proof of Theorem \ref{main_theorem_expectation}]
Lemma \ref{uniform_integrability} states that $\big(X_n-\floor{\log_2n}-\floor{\ln n}\big)^k$ is uniformly integrable and Theorem \ref{main_theorem_corollary} established its convergence in distribution to $\big(\text{\normalfont dGum}(-x-y-c(x))\big)^k$. Together  this implies
\begin{align*}
\mathbb{E}\Big[\big(X_n-\floor{\log_2n}-\floor{\ln n}\big)^k\Big]\to\mathbb{E}\Big[\big(\text{\normalfont dGum}(-x-y-c(x))\big)^k\Big].
\end{align*}
\end{proof}

\subsection{Proof of Lemma \ref{lemma2}}\label{proof_lemma2}
The number of informed nodes, $|I_t|$, fulfils a so-called self-bounding property, for reference see \cite{boucheron2013concentration}. One striking consequence thereof is the following bound.
\begin{lemma}[\cite{Daknama2019}]\label{lemma_var}
For any $t\in \mathbb{N}$,
\begin{align*}
\textnormal{Var}\big[|I_{t+1}|\mid I_t\big]\le\mathbb{E}\big[|I_{t+1}|\mid I_t\big].
\end{align*}
\end{lemma}
This bound on the variance and Chebychev's inequality ensure that the number of informed nodes is highly concentrated around its expectation as soon as enough nodes are informed. Moreover,   even stronger concentration results are possible, as self-bounding functions admit exponential concentration inequalities, see e.g.~\cite{boucheron2013concentration}. Here, Chebychev is sufficient for our application.
\begin{lemma}
\label{concentration_chernoff}
\label{3}
Let $0<c\leq 1$, let $t_0\in \N$ and assume that $|I_{t_0}|\ge n^{c}$.
For $t \in \N$ and $\varepsilon>0$ let $C_t$ denote the event that $$\big||I_{t+1}|-\mathbb{E}_t[|I_{t+1}|]\big|\leq (\mathbb{E}_t[|I_{t+1}|])^{1/2+\varepsilon}.$$ 
Then
\begin{align*}
P_{t_0}\left[\bigcap\limits_{t \geq t_0}C_t\right]= 1-{O}\left (n^{-c\varepsilon}\right ).
\end{align*}
\end{lemma}
\begin{proof}
From \cite[Corollary 3.2]{doerr_tight_2014} it is known that for any $r>0$ 
$$P\big[X_n \geq \ceil*{\log_2n} + \ln n + 2.188 + r\big]\leq 2e^{-r}.$$
 Thus it suffices (with lots of room to spare) to show 
\begin{equation}
\label{eq:tmplog2}
P_{t_0}\left[\bigcup\limits_{t_0\leq t \leq \log ^2n}\overline{C_t}\right] = {O}\left (n^{-3c\varepsilon/2}\right ).
\end{equation} 
By using Chebychev's inequality and Lemma \ref{lemma_var},
\begin{align*}
P_t\big[\overline{ C_t}\big]&=P_t\left [\big||I_{t+1}| - \mathbb{E}_t[|I_{t+1}|]\big|> \mathbb{E}_t[|I_{t+1}|]^{1/2 + \varepsilon}\right ] \leq \frac{\text{Var}[|I_{t+1}|]}{\mathbb{E}_t[|I_{t+1}|]^{1 +2 \varepsilon}}\le  \mathbb{E}_t[|I_{t+1}|]^{-2 \varepsilon}.
\end{align*}
Since $\mathbb{E}_t[|I_{t+1}|] \ge |I_{t+1}| \ge |I_{t_0}|$
the claim follows from~\eqref{eq:tmplog2} and the union bound.
\end{proof}
Lemma \ref{expboundapplied} establishes a connection between the expected value of $|I_{t+1}|$ and our previously defined function $f$, see below Equation~\ref{eq:doubling}. This has also been observed (though not so precise) in \cite{pittel_on_1987} and we include a quick proof for completeness.
\begin{lemma}
\label{expboundapplied}
Let $t\in\N$ and $n\ge 3$. Then
$$f(|I_t|/n)n \le \mathbb{E}_t|I_{t+1}|] \le   f(|I_t|/n)n + 5.$$
\end{lemma}
\begin{proof} 
Each uninformed node $u\in U_t$ remains uninformed if all $|I_t|$ informed nodes do not push to $u$. Since all these events are independent, we obtain that the probability that $u$ remains uninformed is $(1 - 1/(n-1))^{|I_t|}$. Thus by linearity of expectation
\begin{align*}
\mathbb{E}_t[|I_{t+1}|]&=|I_t| + \big(n-|I_t|\big)\left(1-\left(1-\frac{1}{n-1}\right)^{|I_t|}\right)=n- \big(n-|I_t|\big)\left(1-\frac{1}{n-1}\right)^{|I_t|}. 
\end{align*}
For a lower bound we use $1-x\le e^{-x}$  and get
\begin{align*}
\mathbb{E}_t[|I_{t+1}|]\ge n- \big(n-|I_t|\big)e^{-|I_t|/(n-1)}\ge n-\big (n-|I_t|\big)e^{-|I_t|/n}=f(|I_t|/n)n . 
\end{align*}
For an upper bound we use $1-x\ge e^{-x-x^2}$ for all $x\le1/2$
\begin{align*}
\mathbb{E}_t[|I_{t+1}|]&\le n- \big(n-|I_t|\big)e^{-|I_t|/(n-1)-|I_t|/(n-1)^2}\le n- \big(n-|I_t|\big)e^{-|I_t|/n}\exp\left (-\frac{2|I_t|}{(n-1)^2}\right ) 
\end{align*}
and again using $1-x\le e^{-x}$ 
\begin{align*}
\mathbb{E}_t[|I_{t+1}|]&\le n- \big(n-|I_t|\big)e^{-|I_t|/n}\left (1-\frac{2|I_t|}{(n-1)^2}\right )\le f(|I_t|/n)n+5. 
\end{align*}
\end{proof}
Lemma \ref{bound_with_derivative} is an auxiliary result that we use in the proof of Lemma~\ref{6}. It shows that $f$ is concave and has decreasing derivative on the interval $[0,1]$, the stated property is a direct consequence.
\begin{lemma}
\label{bound_with_derivative}
Let $0<x_1\leq x_2<1$. Then $|f(x_2)-f(x_1)|\leq (2-x_1)e^{-x_1}(x_2-x_1)$.
\end{lemma}
\begin{proof}
It is $f'(x)=(2-x)e^{-x}$ and $f''(x)=(x-3)e^{-x}$; in particular, $f'$ is monotonically decreasing and takes only positive values on $[x_1,x_2]$. Furthermore $$\max\limits_{x \in [x_1,x_2]}f'(x)= (2-x_1)e^{-x_1}$$  and therefore, as a direct consequence of the mean value theorem, we have $$|f(x_2)-f(x_1)|\leq (x_2-x_1) \max_{x \in [x_1,x_2]}f'(x) = (2-x_1)e^{-x_1}(x_2-x_1).$$
\end{proof}
We state a simple corollary for later reference.
\begin{corollary}\label{fAdditive}
Let $i\in\N$ and $r,s\in [0,1/2]$. Then
$f^{(i)}(r+s)\le f^{(i)}(r)+2^is.$
\end{corollary}
Having these lemmas as ingredients we can prove the main result of this subsection. 
Lemma \ref{expboundapplied} shows that the expectation of $|I_{t+1}|$ is given by $f(|I_t|/n)n$ and Lemma \ref{concentration_chernoff} shows that $|I_{t+1}|$ is closely concentrated around its expectation in (nearly) all rounds. To then prove that $|I_{t+\tau}|$ is close to $f^{(\tau)}(|I_t|/)n$ for any $\tau\in \mathbb{N}$ we need to make sure that the errors in the concentration and the approximation of the expectation are not blown up by repeated applications of $f$. We will show that $f$ can indeed increase the error in each step by a factor that can be as large as $\sqrt{2}$, but luckily this only happens when $|I_{t+\tau}|=o(n)$ and thus the accumulated error will remain small (as $|I_t|$ nearly doubles in this regime). 
\begin{proof}[Proof of Lemma \ref{6}]
Let  $0<\varepsilon<c/10$, and assume, with foresight, that $n \ge n_0$, where $n_0$ satisfies the inequalities $$\sqrt{2}+10n_0^{-\varepsilon}<\sqrt{2+\varepsilon} \quad \text{and} \quad n_0^c \ge 25.$$ As $T\ge c\log_2 n$ and because of \eqref{eq:doubling}, that is, $|I_t|= 2^t$ for all $t\le \floor{0.49\log_2n}$, we have $|I_T|\ge n^c$. Consequently we can apply Lemma \ref{concentration_chernoff} and thus get with probability $1-O(n^{-c\varepsilon})$
\begin{align}\label{eq:concentration}
\big||I_{t+1}|-\mathbb{E}_{t}[|I_{t+1}|]\big|\le \mathbb{E}_{t}[|I_{t+1}|]^{1/2+\varepsilon},
\quad \text{ for all $t\ge T$}.
\end{align} 
For the rest of this proof we assume that \eqref{eq:concentration} holds. Set $$\alpha_{T+t}= f^{(t)}(|I_T|/n),  \quad t\in\mathbb{N}_0.$$ We will first argue that
\begin{equation}
\label{eq:auxassumption}
\big ||I_{t}|-\alpha_{t}n\big|\le \alpha_{t} ^{1/2+\varepsilon}n^{1/2+2\varepsilon}\sqrt{2+\varepsilon}^{t-T}=:d_{t}.
\end{equation}
for all $t \ge T$ such that $d_{t} \le n^{1-\varepsilon}$. Note that this is obviously true for $t = T$. For the induction step we argue that 
\begin{align}
\label{eq:claim_1}
\big||I_{t+1}|-\alpha_{t+1}n\big|\le  \alpha_{t+1} ^{1/2+\varepsilon}n^{1/2+2\varepsilon}\sqrt{2+\varepsilon}^{t+1-T} = d_{t+1}.
\end{align}
To see this, we  use Lemma~\ref{expboundapplied},~\eqref{eq:concentration} and the fact that $|I_{t+1}| \le 2|I_{t}|$ (in this order) to obtain the bound
\begin{align*}
\big ||I_{t+1}|-f(|I_{t}|/n)n\big|\le \big ||I_{t+1}|-\mathbb{E}_{t}[|I_{t+1}|]\big| +5\le  (2|I_t|)^{1/2+\varepsilon}+5.
\end{align*}
Then we apply Lemma~\ref{bound_with_derivative} to estimate the difference of $f(|I_{t}|/n)$ and $\alpha_{t+1} = f(\alpha_{t})$, and infer  from~\eqref{eq:auxassumption}, using $e^x\le 1+2x$ for all $0\le x\le 1$, that
\begin{align*}
\big| f(|I_{t}|/n)n -\alpha_{t+1}n\big|
& \le \big ||I_{t}|-\alpha_{t}n\big|\big(2-\min\{\alpha_{t},|I_{t}|/n\}\big)e^{-\min\{\alpha_{t},|I_{t}|/n\}} \\
& \le d_{t}\big(2- \alpha_{t} + d_{t}/n\big)e^{-\min\{\alpha_{t},|I_{t}|/n\}} \\
& \le d_{t}\big(2- \alpha_{t}\big)e^{- \alpha_{t} + d_{t}/n}+ d_{t}^2/n\\
& \le
d_{t}(2-\alpha_{t})e^{-\alpha_{t}}+{5d_{t}^2}/{n}.
\end{align*}
All in all we have argued that for all $t$ such that $d_{t}\le n^{1-\varepsilon}$
\begin{equation*}
\begin{aligned}
\big ||I_{t+1}|-\alpha_{t+1}n\big |&\leq\big ||I_{t+1}|-f(|I_{t}|/n)n\big |+\big |f(|I_{t}|/n)n-\alpha_{t+1}n\big |\\
& \le  (2|I_t|)^{1/2+\varepsilon}+5 ~+~ d_{t}(2-\alpha_{t})e^{-\alpha_{t}}+5d_{t}^2/n\\
&\le  2(\alpha_{t}n + d_{t})^{1/2 + \varepsilon} + 5 + d_{t}(2-\alpha_{t})e^{-\alpha_{t}}+5d_{t} n^{-\varepsilon}.
\end{aligned}
\end{equation*}
Our assumptions on $\varepsilon$ and $n$ imply that $d_{t}^{1/2 + \varepsilon} \le d_{t}n^{-\varepsilon}$. Moreover, $\alpha_Tn \ge  n^c \ge 25$, and thus
\begin{equation}
\begin{aligned}
\label{eq:bound_for_I_t+1}
\big ||I_{t+1}|-\alpha_{t+1}n\big |
&\le  3(\alpha_{t}n)^{1/2 + \varepsilon}  + d_{t}(2-\alpha_{t})e^{-\alpha_{t}}+7d_{t} n^{-\varepsilon} \\
& \le d_{t}(2-\alpha_{t})e^{-\alpha_{t}}+10d_{t} n^{-\varepsilon}.
\end{aligned}
\end{equation}
To understand \eqref{eq:bound_for_I_t+1} consider the auxiliary function 
$$H(x)= \sqrt{\frac{f(x)}{x}}-\frac{f'(x)}{\sqrt{2}}=\sqrt{\frac{1-(1-x)e^{-x}}{x}}-\frac{(2-x)e^{-x}}{\sqrt{2}}.$$
 As $(1-x)e^{-x}=1-2x+O(x^2)$ as $x\to 0$ we have that $\lim_{x\to 0} (1-(1-x)e^{-x})/x=2$ and thus $\lim_{x\to 0} H(x)=0$. Furthermore is $H$ an increasing function on the interval $[0,1]$ as, 
 $$H'=\frac{1}{2}\left (\frac{2(1-x)e^{-x}}{x^2}\right )^{-1/2}+\frac{(3-x)e^{-x}}{\sqrt{2}}\ge 0\quad \text{for}\quad x\le 1.$$
Therefore $H(x)\ge 0$ for all $0\le x\le 1$ and consequently, using $\alpha_{t+1}>\alpha_{t}$,
$$\left( \frac{\alpha_{t+1}}{\alpha_{t}}\right )^{1/2+\varepsilon} \ge \left( \frac{\alpha_{t+1}}{\alpha_{t}}\right )^{1/2} \geq \frac{(2-\alpha_{t})e^{-\alpha_{t}}}{\sqrt{2}}.$$
Since $d_{t} = \alpha_{t}^{1/2 + \varepsilon}n^{1/2+2\varepsilon}\sqrt{2+\varepsilon}^{t-T}$, applying the previous bound to~\eqref{eq:bound_for_I_t+1} implies \eqref{eq:claim_1} for all $n\ge n_0$, that is, all $n$ such that $\sqrt{2}+10n^{-\varepsilon}<\sqrt{2+\varepsilon}$. This completes the induction step and the proof of~\eqref{eq:auxassumption} is completed.

Actually our arguments yield also the following statement, which is stronger than~\eqref{eq:auxassumption} when there are ``many'' informed nodes. In particular,
for all $t'\in\N$ such that $(2-\alpha_{t'})e^{-\alpha_{t'}}<1-\varepsilon$ Equation~\eqref{eq:bound_for_I_t+1} also yields for all $n\ge n_0$ 
$$\big||I_{T+t'}|-\alpha_{T+t'}n\big|\le  d_{t'} \quad \Rightarrow\quad \big||I_{T+t'+1}|-\alpha_{T+t'+1}n\big|\le  d_{t'},$$
meaning that the absolute error does not increase any more after round $t'$. (Actually the error decreases by a factor of at least $\varepsilon$ after that round, but we do not need this.) To complete the proof we show that we can choose $t'$ such that $d_{t'} \le n^{1-c/4}$ and $(2-\alpha_{t'})e^{-\alpha_{t'}}<1-\varepsilon$. To this end, consider
$$T'=\floor{\log_2 n}-4 -T$$
and applying Lemma \ref{lemma1} to $\alpha_{T'}$ yields
\begin{align*}
\alpha_{T+T'}=f^{(T')}(|I_T|/n)\ge f^{(\floor{\log_2 n}-4-T)}\left (2^{T}/n\right )\ge 2^{-4}\big(1-2^{-8+1}\big)
\end{align*}
and furthermore, a simple computation yields that $\alpha_{T+T'+5} \ge 3/4$.
Thus  $$(2-\alpha_{T+T'+5})e^{-\alpha_{T+T'+5}}\le (2-3/4)e^{-3/4}<1-\varepsilon$$ and we set $t' := T'+5$. Moreover,
\begin{align*}
 d_{t'} \le n^{1/2+2\varepsilon}\sqrt{2+\varepsilon}^{t'}
 \le (2+\varepsilon)n^{1/2+2\varepsilon+(1-c)\log_2(2+\varepsilon)/2}.
\end{align*}
Note that $\log_2(2+\varepsilon) \le 1 + \varepsilon$. Plugging this into the exponent yields that if $\varepsilon<c/10$ and $n$ is large enough then $d_{t'} \le n^{1-c/4} (\le n^{1-\varepsilon})$, as claimed.
\end{proof} 
\subsection{Proof of Lemma \ref{lemma1}}\label{proof_lemma1}
We begin with showing the basic inequality
\begin{equation}
\label{eq:fbounds}
	2x(1-x) \le f(x) \le 2x.
\end{equation}
To see this, note that $e^{-x} \le 1 - x + x^2/2$ for $x\in[0,1]$ and so
\begin{align*}
f(x)&=1-e^{-x}(1-x)\geq 1-\left(1-x+\frac{x^2}{2}\right)(1-x) \geq x\left(2-\frac{3}{2}x\right)\geq 2x-2x^2,
\end{align*}
which establishes the first inequality in~\eqref{eq:fbounds}. The other inequality follows directly  from the simple bound $e^{-x} \ge 1 - x$.

Let us write $z_0 = 2^{t_0} / n$ and $z_i = f(z_{i-1}) = f^{(i)}(z_0)$; we want to bound $z_{t_1-t_0}$, where $t_1 = \floor*{\log_2 n} - a$ and $t_0 \le\floor*{0.49 \log_2n}$. Clearly $z_i \le 2^i z_0$, which shows the upper bound in Lemma \ref{lemma1}. Using~\eqref{eq:fbounds} we obtain by induction
\[
	z_i \ge 2^i z_0\cdot \prod_{j=0}^{i-1}(1 - 2^j z_0), \quad  i\in \mathbb{N}.
\]
Further, using the bound $1 - x \ge e^{-x-x^2/2(1-x)}$, valid for any $x\in[0,1)$ we obtain
\[
	z_i \ge  2^i z_0\cdot \exp\left\{-z_0 \sum_{0 \le j < i }2^j - z_0^2\sum_{0 \le j < i }\frac{4^j}{2(1-2^jz_0)}\right\}
\]
Note that our assumptions guarantee that $2^{t_1-t_0}z_0 = 2^{-a} < 0.1$, and so for any $1 \le i \le t_1-t_0$
\[
	z_i \ge  2^i z_0\cdot \exp\left\{-2^i z_0 - (2^iz_0)^2\right\} \ge 2^i z_0\cdot(1 - 2^{-a} - 2^{-2a}).
\]
Finally note that $1-y-y^2 \ge 1-2y$ for any $y \in [0,1]$, and so the last term is bounded by $2^i z_0 \cdot(1 - 2^{-2a+1})$, which coincides with the lower bound claimed in~Lemma \ref{lemma1}.
\begin{corollary}\label{boundf}
For all $x\in[0,1]$ and $i\in \N$
$$2^{i}x\big(1-2^ix-2^{2i}x^2\big)\le f^{(i)}(x)\le 2^ix.$$
\end{corollary}
\subsection{Proof of Lemma \ref{lemma3}}\label{proof_lemma3}


A main tool in the fortcoming proof is the following result, which  states that a sum of normalized independent geometric random variables converges to a Gumbel distributed random variable. 
\begin{theorem}[{\cite{erdos1961classical}}]
\label{erdos_renyi_gumbel}
Let $T_1,\dots, T_{n-1}$ be independent random variables such that 
 $T_i \sim \textnormal{Geo}((n-i)/(n-1))$ for $1 \le i < n$.  Then, in distribution
$$n^{-1}\sum \limits_{1 \le i < n}\big(T_i-\mathbb{E}[T_i]\big)\ra \textnormal{Gum}(\gamma).$$
\end{theorem}
Unfortunately we can not  apply directly Theorem \ref{erdos_renyi_gumbel} to our setting, as we will have to deal with a sum of independent geometric random variables that are not normalized with the `correct' factor $n^{-1}$. However, the next well-known statement  assures that if the error is small enough we will still converge to the same limiting distribution
\begin{theorem}[{Slutsky's Theorem, see, e.g., \cite[p. 19]{serfling2009approximation}}]
\label{product_of_random_variables_conv_in_distr}
Let $(X_n)_{n \in \mathbb N}$, $(Y_n)_{n \in \mathbb N}$ and $(Z_n)_{n \in \mathbb N}$ be sequences of real-valued random variables. Suppose that $X_n\ra X$ in distribution  and that there are constants $a, b \in \R$ such that $Y_n \ra a$ and $Z_n\ra b$ in probability. Then $Y_n X_n + Z_n \ra a X + b$ in distribution.
\end{theorem}
We now show a more general version of Theorem \ref{erdos_renyi_gumbel} that is applicable to our setting.
\begin{lemma}
\label{thm_approximate_gumbel_sum_of_deviations}
Let $T_1,\dots, T_{n-1}$ be independent random variables such that 
 $T_i \sim \textnormal{Geo}((n-i)/(n-1))$ for $1 \le i < n$. Let furthermore $\varepsilon>0$ and $s:\mathbb{N}\to [1,n]$ be a function such that $s(n-i) \geq \big(1-o(1)\big)(n-c\cdot i)$ for any positive integer $i <\varepsilon n.$
Then, in distribution
\begin{align*}
\sum\limits_{(1-\varepsilon) n\le i < n} \frac{T_i-\mathbb{E}[T_i]}{s(i)}  \ra \textnormal{Gum}(\gamma) .
\end{align*}
\end{lemma}
\begin{proof}
Let  $D_i=T_i-\mathbb{E}[T_i]$ be the centralised version of $T_i$. Then
\begin{align*}
\sum\limits_{(1-\varepsilon) n\le i < n} \frac{D_i}{s(i)}=\sum\limits_{1\le i<n} \frac{D_i}{n}-\sum_{1\le i<(1-\varepsilon)n} \frac{D_i}{n}+\left (\sum_{(1-\varepsilon) n\le i < n} \frac{D_i}{s(i)}-\sum_{(1-\varepsilon) n\le i < n} \frac{D_i}{n}\right ).
\end{align*}
A direct applicaition  of Theorem \ref{erdos_renyi_gumbel} guarantees that the first sum converges to Gum$(\gamma)$ in distribution.
To complete the proof it is sufficient to argue that in probability
\begin{align}
\label{eq:tmpzeroterms}
  \sum_{1\le i<(1-\varepsilon)n} \frac{D_i}{n} \ra 0
 \quad\text{and}\quad 
  \sum_{(1-\varepsilon) n\le i < n} \left (\frac{D_i}{s(i)}- \frac{D_i}{n}\right )\ra 0,
\end{align}
from which the claim in the lemma follows immediately from Theorem \ref{product_of_random_variables_conv_in_distr}. Since the $D_i$'s are centralised
$$\mathbb{E}\left [\sum_{1\le i<(1-\varepsilon)n} \frac{D_i}{n}\right ]=0,$$
and using that $\text{Var}[T_i]=\big((n-1)(i-1)\big)/(n-i)^2$ for all $i<n$
$$\text{Var}\left [\sum_{1\le i<(1-\varepsilon)n} \frac{D_i}{n}\right ]=\sum_{1\le i<(1-\varepsilon)n} \frac{\text{Var}[T_i]}{n^2}=\sum_{1\le i<(1-\varepsilon)n} \frac{1}{n^2}\frac{(n-1)(i-1)}{(n-i)^2}\le \sum_{1\le i<(1-\varepsilon)n} \frac{1}{(\varepsilon n)^2}=o(1).
$$
Thus Chebychev's inequality directly implies that $$\sum_{1\le i<(1-\varepsilon)n}\frac{D_i}{n} \ra 0\quad \text{in probability}.$$
It remains to treat the second term in~\eqref{eq:tmpzeroterms}. 
We compute the variance as before
\begin{align*}
\text{Var}\left [\sum_{(1-\varepsilon) n\le i < n} \frac{D_i}{s(i)} -\frac{D_i}{n}\right ]&=
\sum_{(1-\varepsilon) n\le i < n}\Big (\frac{1}{s(i)}-\frac{1}{n}\Big )^2\frac{(n-1)(i-1)}{(n-i)^2}
\le \sum\limits_{1\le i\le\varepsilon n}\Big(\frac{1}{s(n-i)}-\frac{1}{n}\Big)^2\frac{n^2}{i^2}.
\end{align*}
However, this is also $o(1)$, 
as  $s(n-i)\ge \big(1-o(1)\big)(n-c\cdot i)$ for all integers $i\le \varepsilon n$ by assumption, and therefore
\begin{align*}
0\le \frac{1}{s(i)}-\frac{1}{n}\le\frac{1}{(1+o(1))(n-c\cdot i)}-\frac{1}{n}= (1+o(1))\frac{c\cdot  i+o(n)}{n^2}, \quad i\le \varepsilon n.
\end{align*}
In summary we have shown that 
\begin{align*}
\text{Var}\left [\sum_{(1-\varepsilon) n\le i < n} \left (\frac{D_i}{s(i)}- \frac{D_i}{n}\right )\right ]=o(1)\quad \text{and clearly }\quad \mathbb{E}\left [\sum_{(1-\varepsilon) n\le i < n} \left (\frac{D_i}{s(i)}- \frac{D_i}{n}\right )\right ]=0.
\end{align*}
Thus Chebychev's inequality implies also the second statement in~\eqref{eq:tmpzeroterms} and the proof is complete.
\end{proof}
A further ingredient that we shall exploit is the following fact. If a sequence of random variables $X_n \ra X$ in distribution with distribution functions $F_n\to F$ and if $F$ is continuous everywhere, then the convergence of $F_n$ to $F$ is even uniform.
\begin{theorem}[{Polya's Theorem, {\cite[Theorem 1]{polya1920}}}]
\label{9}
For each $n \in \N$ let $X_n$ be a real-valued random variable with distribution function $F_n$. Assume that $X_n \ra X$ in distribution. If $X$ has continuous distribution function $F$, then
$$\lim\limits_{n \ra \infty}\sup\limits_{x \in \mathbb R} |F_n(x) - F(x)| =0.$$
\end{theorem}
We need one more auxiliary lemma that gives an upper bound on the informed nodes when going one round backwards in order to later convert the number of Coupons into the number of rounds that are needed to finish the protocol.
Appropriately, Lemma \ref{concentration_chernoff}  assures that in all rounds the number of informed nodes is tightly concentrated around its expectation, which in turn is described by $f$, thus applying $f^{-1}$ will give a good bound.
\begin{lemma}
\label{11}
Let $t_0\in \mathbb{N}$ and $0<\varepsilon < 1/6.$ Let $C_t$ be the event that $\big||I_{t+1}|-E_t[|I_{t+1}|]\big|\leq (E_t[|I_{t+1}|])^{1/2+\varepsilon}$, as given in Lemma \ref{concentration_chernoff}. Then  for $n$ large enough the event $\bigcap_{t \geq t_0}C_t$ implies   for all $t \ge t_0$
$$|I_{t}| \geq \big(1-n^{-1/3}\big)\cdot e\cdot \big (|I_{t+1}| - (1-1/e)n\big) .$$
\end{lemma}
\begin{proof}
Lemma \ref{expboundapplied} and $C_t$ together give that
\begin{align*}
|I_{t+1}| &\leq E_{t}[|I_{t+1}|]|+ (E_t[|I_{t+1}|])^{1/2+\varepsilon} = f(|I_{t}|/n)n + o\big(n^{2/3}\big) .
\end{align*} 
Using the definition of $f(x)=1-(1-x)e^{-x}$ and that $|I_t|\le n$ for all $t$ we get that 
\begin{align*}
|I_{t+1}| &\leq  n-e^{-|I_{t}|/n}(n-|I_{t}|)+o\big(n^{2/3}\big) \leq (1-1/e)n+ |I_{t}|/e + o\big(n^{2/3}\big).
\end{align*}
Rearranging yields the claimed statement.
\end{proof}
Let us briefly outline the proof of Lemma \ref{lemma3}. We have already shown bounds for the number of informed nodes after $\floor{\log_2 n} -a+b$ rounds in~\eqref{eq:start_main}. Starting from these bounds we will use the Coupon Collector Problem to compute the number of \emph{pushes} that are needed to inform all remaining uninformed nodes. This will yield sums of independent geometric random variables (one summand for each uninformed node). Using Lemma \ref{11} we will translate these numbers of pushes into numbers of rounds, which results in an almost correctly normalised sum of geometric random variables that Lemma \ref{thm_approximate_gumbel_sum_of_deviations} assures to converge to a Gumbel distribution. We will end up with upper and lower bounds to the distribution function of \emph{push}.
\begin{proof}[Proof of Lemma \ref{lemma3}] In this proof we will establish a connection between the Coupon Collector Problem and the behavior of \push. Let $v \in V$ be the node that was initially informed. Instead of every informed node choosing one of its neighbours uniformly at random, we now assume that it samples one node in $V\setminus \{v\}$ uniformly at random. This defines an equivalent model, as for all $u\in V$ the probability to choose any specific node in $V\setminus\{u,v\}$ does not change (it equals $1/(n-1)$ in both models) and choosing $u$ or $v$ makes no difference for the distribution of the set of informed nodes. Thus \push is the same as drawing coupons out of a pool of $n-1$ different coupons, but doing so in batches with size being the number of distinct coupons already collected plus one, the `plus one' representing the initially informed node $v$.
It is widely known and easy to see that assuming $1\le i\le n-1$ coupons (including $v$) have already been collected, then
\begin{equation}
\label{Y_i_are_independent_geometric_random_variables}
T_i \sim \text{Geo}\left(\frac{n-i}{n-1}\right), \quad 1\le i\le n-1.
\end{equation}
describes the number of coupons one needs to draw in order to draw the next, $(i+1)$st new, distinct coupon. Thus in order to collect all $n$ coupons one needs to draw $\sum_{i=1}^{n-1}T_i$ coupons, where the summands are independent random variables. However, we are not particularly interested in the total number of coupons drawn, but in the number of batches needed. If a batch has size $s\le n-1$, then this batch is worth $s$ coupons, or vice versa, each coupon drawn in this batch is worth $1/s$ batches. Thus we need to estimate the size of the batch that contained all coupons that were needed to draw the $(i+1)$st distinct coupon, or if these coupons were contained in multiple batches, then we bound all those involved -- we call these batches the batches that are \emph{linked to~$i+1$}. Let $L_i$ be the smallest and $U_i$ the largest size of a batch linked to the $(i+1)$st  coupon. Then certainly $U_i\le i$, as at the time that the $(i+1)$st distinct coupon gets collected there are obviously at most $i$ distinct collected coupons.
Using our assumption $\ell \cdot n\le |I_{\floor{\log_2 n}-a+b}|\le u\cdot  n$ we thus obtain
\begin{align}\label{eq:geom_bound}
\ceil*{\sum\limits_{i=\floor*{un}}^{n-1} \frac{T_i}{U_i}} \leq X_n-\big(\floor{\log_2 n}-a+b\big) \leq \ceil*{\sum\limits_{i=\ceil*{\ell n}}^{n-1} \frac{T_i}{L_i}}.
\end{align}
Abbreviating $Y_n=X_n-(\floor{\log_2 n}-a+b)$ and recalling that $U_i\le i$ yields
\begin{align*}
 Y_n \ge \ceil*{\sum\limits_{i=\floor*{un}}^{n-1} \frac{T_i}{U_i}}= \ceil*{\sum\limits_{i=\floor*{un}}^{n-1} \frac{T_i}{i}}.
\end{align*}
As the $T_i$ are independent and geometrically distributed, we can apply Lemma \ref{thm_approximate_gumbel_sum_of_deviations} and for $G\sim  \text{Gum}(\gamma)$ we obtain with Theorem \ref{9}
\begin{align*}
\sup_{k \in \mathbb{Z}} \left| P\bigg [ \sum\limits_{i=\floor*{un}}^{n-1} \frac{T_i - \mathbb{E}[T_i]}{i}\ge k\bigg] - P\left [ G\ge k\right ] \right| = o(1)
\end{align*}
and therefore 
\begin{align*}
\ceil*{\sum\limits_{i=\floor*{un}}^{n-1} \frac{T_i}{i}} =  \ceil*{\sum\limits_{i=\floor*{un}}^{n-1}\frac{\mathbb{E}[T_i]}{i} + \sum\limits_{i=\floor*{un}}^{n-1} \frac{T_i - \mathbb{E}[T_i]}{i}} \succsim\ceil*{\sum\limits_{i=\floor*{un}}^{n-1}\frac{1}{i(1-i/n)} + G}.
\end{align*}
The partial fraction decomposition
$\big(i(1-i/n)\big)^{-1}=(n-i)^{-1}+i^{-1}$
allows us to simplify
\begin{align*}
\ceil*{\sum\limits_{i=\floor*{un}}^{n-1}\frac{1}{i(1-i/n)} + G}=
\ceil*{\sum\limits_{i=\floor*{un}}^{n-1} \frac{1}{n-i} + \sum\limits_{i=\floor*{un}}^{n-1} \frac{1}{i} + G}  =
\ceil*{\sum\limits_{i=1}^{n-\floor*{un}}\frac{1}{i}  +  \sum\limits_{i=\floor*{un}}^{n-1}\frac{1}{i} + G}.
\end{align*}
Expressing these partial harmonic sums using the asymptotic expansion for the $n$th harmonic number
\begin{equation}
\label{approximation_for_nth_harmonic_number}
\sum\limits_{1\le k \le n} k^{-1}=H_n= \ln n + \gamma + {O}(1/n) 
\end{equation}
we get, using Lemma \ref{finishGumbel}, 
\begin{align*}
Y_n &\succsim
\ceil*{\ln(n-un) + \gamma + \ln n + \gamma - \ln(un) - \gamma + G + {O}(1/n) } \nonumber \\ &= 
\ceil*{\ln n + \ln\left(\frac{n(1-u)}{un}\right)+\gamma + G+{O}(1/n)} \nonumber \\&\succsim
 \ceil*{\ln n +\ln(1/u-1) +  \gamma + G }. 
\end{align*} 
We now look at the upper bound in \eqref{eq:geom_bound}. For all $\floor{\ell n}\le i\le n-1$ we specify an appropriate bound for $L_i$. To obtain it, assume that $t$ is the round in which the $i$th vertex was informed. Then all batches that are linked to the $(i+1)$st coupon have size at least $|I_t|$, i.e.~$L_i\ge |I_t|$, as the $(i+1)$st distinct coupon  is drawn after the $i$th distinct coupon, i.e., it cannot be drawn in any round $t'< t$. However, we do not know $|I_t|$, but we certainly can say that $|I_{t+1}| \ge i$. So, Lemma~\ref{11}, guarantees that whp
\begin{align*}
|I_t|\ge \big(1-n^{-1/3}\big)  \cdot e\cdot \big(i-(1-1/e)n\big)  \quad \text{ for all } i \in \{\floor*{\ell n}, \dots, n-1\}.
\end{align*} 
(Note that $t = t(i)$ in that statement.) In particular, whp
\begin{align*}
L_i \ge |I_t| \ge \big(1-n^{-1/3}\big)\cdot (n - e\cdot (n-i)) \quad \text{ for all } i \in \{\floor*{\ell n}, \dots, n-1\}.
\end{align*}
Let $C$ be the event that Lemma \ref{11} conditions on, that is that $|I_t|$ (for all $t\in\N$) is closely concentrated around its expectation. Let $k\in \mathbb{N}$ and $$B=\left\{\ceil*{\sum\limits_{i=\ceil*{\ell n}}^{n-1} \frac{T_i}{L_i}}\ge k\right \}.$$
Then $P(B)=P(C\cap B)+o(1)$ and as $$\left \{C\cap B\ge k\right \}\Rightarrow\left \{\ceil*{\sum\limits_{i=\ceil*{\ell n}}^{n-1} \frac{T_i}{(1-n^{-1/3}) (n-e\cdot (n-i))}}\ge k\right\}$$ we get, recalling $Y_n=X_n-(\floor{\log_2 n}-a+b)$, that
\begin{align*}
 Y_n \leq \ceil*{\sum\limits_{i=\ceil*{\ell n}}^{n-1} \frac{T_i}{L_i}}\precsim \ceil*{\sum\limits_{i=\ceil*{\ell n}}^{n-1} \frac{T_i}{(1-n^{-1/3}) (n-e\cdot(n-i))}}.
\end{align*}
Again applying Lemma \ref{thm_approximate_gumbel_sum_of_deviations} and Theorem \ref{9}, for $G \sim \text{Gum}(\gamma)$ and $c=e$, we obtain
\begin{equation*}
\begin{split}
Y_n &\precsim \ceil*{\sum\limits_{i=\ceil*{\ell n}}^{n-1} \frac{\mathbb{E}[T_i]}{(1-n^{-1/3}) (n-e\cdot(n-i))}+ \sum\limits_{i=\ceil*{\ell n}}^{n-1}\frac{T_i-\mathbb{E}[T_i]}{(1-n^{-1/3})( n-e\cdot(n-i))} }\\ 
&\precsim 
\ceil*{\big(1+{O}(n^{-1/3})\big)\sum\limits_{i=\ceil*{\ell n}}^{n-1} \frac{1}{(n-e\cdot(n-i))(1-i/n)}+ G}. 
\end{split}
\end{equation*}
Let $c = 1 - 1/e$. Using that
${\big((n-e\cdot(n-i))(1-i/n)\big)^{-1}}=({n-i})^{-1}+{(i-cn)}^{-1}$ gives
\begin{align*}
Y_n\precsim \ceil*{\big(1+{O}(n^{-1/3})\big)\left(\sum\limits_{i=\ceil*{\ell n}}^{n-1} \frac{1}{n-i} + \sum\limits_{i=\ceil*{\ell n}}^{n-1} \frac{1}{i-cn}\right) + G}.  
\end{align*}
Using index shifts, the asymptotic expansion for the harmonic number (\ref{approximation_for_nth_harmonic_number}) and Lemma \ref{finishGumbel} yields
\begin{equation*}
\begin{split}
Y_n  & \precsim \ceil*{\big(1+{O}(n^{-1/3})\big)\left(\sum\limits_{i=1}^{n-\ceil*{\ell n}}\frac{1}{i} + \sum\limits_{i=\ceil*{\ell n}-\floor*{cn}}^{n-1-\floor*{cn}}\frac{1}{i}\right) + G + o(1)} \\ & \precsim
\ceil*{\ln n + \ln(1/\ell-1) - \ln(1/\ell) + \gamma + \ln\left(\frac{1-c}{\ell-c}\right)+ G}.  
\end{split}
\end{equation*}
\end{proof}
\subsection{Proof of Lemma \ref{c_is_continuous_push_complete}}
\label{proofs_of_the_preparation_results_push_exact_a}
In this subsection we investigate the double limit 
\begin{align*}
\lim\limits_{a \ra \infty, a \in \N} \lim\limits_{b \ra \infty, b \in \N} -a + b + \ln\big(g^{(b)}(1-2^{-a-x})\big)-x
\end{align*}
where $g(x)=xe^{x-1}$. We will show that this limit exists and defines a continuous function $c(x)$. It being periodic with period 1 is an immediate consequence of substituting $a\to a+1$ in the limit. A similar proof would also yield that $c$ is continuously differentiable, but we only need continuity in the proof of our main theorem. 

Before we actually prove Lemma \ref{c_is_continuous_push_complete} let us state
two auxiliary statements first. In Definition \ref{note_exp_fast_convergence}, we quantify ``exponentially fast convergence'' and in Lemma \ref{exp_fast_conv_to_zero_lemma} we state some simple properties.
\begin{definition}[Exponentially fast convergence]
\label{note_exp_fast_convergence}
Let $(a_n)_{n \in \N}$ be a real-valued sequence and let $c \in (0,1)$. If there is an $n_0 \in \N$ such that for all $n \geq n_0$ we have $|a_{n+1}|<c|a_n|$, then we say that $a_n$ \emph{converges exponentially fast to zero at rate $c$ with start number $n_0$}. 
\end{definition}
\begin{lemma}
\label{exp_fast_conv_to_zero_lemma}
\begin{itemize}
\item[a)] Let $c \in (0,1)$ and let $(a_n)_{n \in \N}$ be a real-valued sequence that converges exponentially fast to zero at rate $c$. Then $\sum_{n\geq 1}a_n$ converges absolutely. 
\item[b)] Let $c\in(0,1)$, $n_0\in \N$ and let $(h_n)_{n\in \N}$ denote a sequence of functions with $h_n:[0,1]\ra \R$ such that for any $x \in [0,1]$ the sequence $(h_n(x))_{n \in \N}$ converges exponentially fast to zero at rate $c$ with start number at most $n_0$. Define $h:[0,1] \ra \R$ by $h(x):=\sum_{n\geq 1}h_n(x)$. Then the sequence of functions $(\sum_{j=1}^nh_j)_{n \in \N}$ converges uniformly to $h$, i.e. 
$$\lim\limits_{n \to \infty}\sup\limits_{x \in [0,1]}\left |h(x)-\sum_{j=1}^nh_j(x)\right |=0.$$  
\end{itemize}
\end{lemma}
\begin{proof}
a) is elementary.  We prove $b)$. Let $\varepsilon>0$. We show that there is an $n_1 \in \N$ such that for all $n\geq n_1$ and for all $x \in [0,1]$ it holds $\left |\sum_{j=1}^n h_j(x)-h(x)\right |< \varepsilon$. For $n \geq n_0$ it is 
$$\left |\sum_{j=1}^nh_j(x)-h(x)\right |=\bigg|\sum\limits_{j=n+1}^{\infty}h_n(x)\bigg |\leq \sum\limits_{j=n+1}^{\infty}|h_n(x)|\leq |a_{n_0}|\sum\limits_{j=n+1}^{\infty}c^j=|a_{n_0}| \frac{c^{n+1}}{1-c}$$ which implies that an $n_1$ as required exists.   
\end{proof}
\begin{proof}[Proof of Lemma \ref{c_is_continuous_push_complete}]
We show first, that for $a$ fixed and any $x\in[0,1]$ the limit $$\lim\limits_{b \ra \infty, b \in \N} b + \ln\left(g^{(b)}\left(1-2^{-a-x}\right)\right)$$ exists and the convergence is uniform. Inductively we get 
\begin{align}
 b + \ln\left(g^{(b)}\left(1-2^{-a-x}\right)\right) &= 
 b + \ln\left (g^{(b-1)}\left (1-2^{-a-x}\right )\right ) + g^{(b-1)}\left (1-2^{-a-x}\right ) - 1 \nonumber \\   
&= 1 + \ln\left (1-2^{-a-x}\right ) - 2^{-a-x} + \sum\limits_{j=1}^{b-1} g^{(j)}\left (1-2^{-a-x}\right ) \label{alternative_representation_for_c}
\end{align} 
which, according to Lemma \ref{exp_fast_conv_to_zero_lemma} $a)$, converges for $b\ra \infty$ because $g^{(j)}\left (1-2^{-a-x}\right )$ converges exponentially fast to zero at rate at most $\exp(-2^{-a-1})<1$ and start number $1$ for $j \ra \infty$ in the sense of Definition \ref{note_exp_fast_convergence}. For $x \in [0,1]$, according to Lemma \ref{exp_fast_conv_to_zero_lemma} $b)$, the convergence is even uniform with respect to $x$. By the Uniform Limit Theorem we thus showed that 
\begin{align}
\label{sequence_gamma_a_converges_push_complete}
\gamma_a(x)=-a + \sum\limits_{j \ge 1}g^{(j)}(1-2^{-a-x}) \quad \text{is continuous for } a \in\mathbb{N}. 
\end{align}
To complete the proof we show that the sequence of continuous functions $(\gamma_a)_{a\in \N}$ converges uniformly. But first we make an observation.  Let $a'>a\in\N$ and $x\in[0,1]$, then, using $g(x)=1-f(1-x),$
\begin{align*}
\gamma_{a'}(x)&=-a'+\sum\limits_{j\ge 1} g^{(j)}(1-2^{-a'-x})=-a'+\sum\limits_{j=1}^{a'-a} g^{(j)}\big(1-2^{-a'-x}\big)+\sum\limits_{j\ge 1}^{\infty} g^{(j)}\left (g^{(a'-a)}\big(1-2^{-a'-x}\big)\right )\\
&=-a-\sum\limits_{j=1}^{a'-a} f^{(j)}\big(2^{-a'-x}\big)+\sum\limits_{j\ge 1}  g^{(j)}\left (1-f^{(a'-a)}\big(2^{-a'-x}\big)\right ).
\end{align*}
Furthermore, we can bound the repeated application of $f$ using Corollary \ref{boundf} and therefore
\begin{align*}
0\le \sum\limits_{j=1}^{a'-a} f^{(j)}\big(2^{-a'-x}\big)\le 2^{-a-x+1} ~\text{and}~ 2^{-a-x}\big(1-2^{-a-x}-2^{-2a-2x}\big) \le f^{(a'-a)}\big(2^{-a'-x}\big)\le2^{-a-x}.
\end{align*}
Thus there is $x'\in [0,1]$ such that $|x-x'|\le 2^{-a}$ and $\gamma_{a'}(x)=\gamma_a(x')+{O}\big(2^{-a}\big).$

With this at hand we show uniform convergence of $(\gamma_a)_{a\in\N}$. In particular, for any $0<\varepsilon<1/8$ we will show that there is some $N\in \N$ such that $\sup_{x\in [0,1]}|\gamma_a(x)-\gamma_{a'}(x)|\le \varepsilon$ for all $a' > a>N$. To achieve this we use our previous observation and obtain that
\begin{align*}
\sup_{x\in [0,1]}|\gamma_a(x)-\gamma_{a'}(x)|&\le \sup_{x\in [0,1],|x-x'|\le 2^{-a}}|\gamma_a(x)-\gamma_{a}(x')|+{O}\big(2^{-a}\big)\\
&= \sup_{x\in [0,1],|x-x'|\le 2^{-a}}\left |\sum_{j \ge 1} \big(g^{(j)}\big(1-2^{-a-x'}\big)-g^{(j)}\big(1-2^{-a-x}\big)\big)\right |+O\big(2^{-a}\big).
\end{align*}
We bound this sum by splitting it into three parts. There is $M_1\in\N$ such that for any $a>M_1 $ there is $N_1\in \N$ ($N_1$ depending on $a$ and $\varepsilon$) such that 
\begin{align}\label{eq:N1}
\varepsilon\le f^{(N_1)}\big (2^{-a-1}\big)\le f^{(N_1+1)}\big (2^{-a}\big)\le 8\varepsilon.
\end{align}
That is, $N_1$ is the number of iterations such that $f^{(N_1)}\big(2^{-a}\big)\approx \varepsilon$, in particular $N_1\le a$, as $f^{(a)}\big (2^{-a-1}\big)\ge 1/8$ by Corollary~\ref{boundf} and the fact that $f$ is increasing. Furthermore, using again that $g^{(j)}\left (1-2^{-a-x}\right )$ converges exponentially fast to zero with rate at most $\exp(-2^{-a-1})<1$ for $j \ra \infty$, there is $c\in\N$ depending only on $\varepsilon$ such that for $N_2:= N_1+c$ 
\begin{align}\label{eq:N2}
0\le \sup_{x\in [0,1]} \sum\limits_{j\ge N_2} g^{(j)}\big(1-2^{-a-x}\big)\le \varepsilon\quad \text{for all } a>M_1.
\end{align}
Then, abbreviating $h^{(j)}=g^{(j)}\big(1-2^{-a-x'}\big)-g^{(j)}\big(1-2^{-a-x}\big)$, we can write
\begin{align}\label{eq:splitSum}
\sum\limits_{j=1}^{\infty}\big(g^{(j)}\big(1-2^{-a-x'}\big)-g^{(j)}\big(1-2^{-a-x}\big)\big)=\sum\limits_{j=1}^{N_1}h^{(j)}+\sum\limits_{j=N_1+1}^{N_2}h^{(j)}+\sum\limits_{j>N_2} h^{(j)}.
\end{align}
In the rest of the proof estimate these sums individually, starting with the first one. Again using \eqref{alternative_representation_for_c} and $f(x)=1-g(1-x)$ we have as $a\to\infty$
\begin{align*}
\sum\limits_{j=1}^{N_1}g^{(j)}\big(1-2^{-a-x'}\big)-g^{(j)}&\big(1-2^{-a-x}\big)=\ln g^{(N_1)}\big(1-2^{-a-x'}\big)- \ln g^{(N_1)}\big(1-2^{-a-x}\big)+{O}(2^{-a})\\
&=\ln \left (1- f^{(N_1)}\big (2^{-a-x'}\big)\right )- \ln\left ( 1-f^{(N_1)}\big (2^{-a-x}\big)\right )+O(2^{-a}).
\end{align*}
By our choice of $N_1$, see~\eqref{eq:N1}, and the elementary inequalities $z/(1+z)\le \ln (1+z)\le z$ for all $z>-1$ this yields the upper bound
\begin{align}\label{eq:firstUpper}
\sup_{x\in[0,1]}\left |\sum\limits_{j=1}^{N_1}h^{(j)}\right |\le \varepsilon+\frac{8\varepsilon}{1+8\varepsilon}+{O}(2^{-a})\quad \text{for all } a>M_1.
\end{align}
We continue with the second sum in \eqref{eq:splitSum}.
Corollary \ref{fAdditive} yields
\begin{equation*}
\begin{aligned}
\left |\sum_{j=N_1+1}^{N_2}h^{(j)}\right |
&=\left |\sum_{j=N_1+1}^{N_2}\big(f^{(j)}\big(2^{-a-x}\big) -f^{(j)}\big(2^{-a-x'}\big)\big)\right |
\le \left( 2^{-a-x}-2^{-a-x'}\right )\sum_{j=N_1+1}^{N_2}2^j.
\end{aligned}
\end{equation*}
Thus, as $N_1\le a$ and $N_2=N_1+c$, where $c$ depends on $\varepsilon$ only, and our assumption $|x-x'|\le 2^{-a}$ there is $M_2\ge M_1$ such that 
\begin{equation}
\begin{aligned}\label{eq:secondUpper}
\left |\sum_{j=N_1+1}^{N_2}h^{(j)}\right |&\le \left (2^{2^{-a}}-1\right )\cdot 2^{-a}\cdot \sum_{j=N_1+1}^{N_2}2^j\le \left (2^{2^{-a}}-1\right )\cdot2^{c+1}\le \varepsilon\quad \text{for all }a>M_2.
\end{aligned}
\end{equation}
In summary, \eqref{eq:splitSum} gives 
\begin{equation*}
\begin{aligned}
\sup_{x\in [0,1]}\big|\gamma_a(x)-\gamma_{a'}(x)\big|&\le \sup_{x\in [0,1],|x-x'|\le 2^{-a}}\left |\sum\limits_{j=1}^{N_1}h^{(j)}+\sum\limits_{j=N_1+1}^{N_2}h^{(j)}+\sum\limits_{j>N_2} h^{(j)}\right |+{O}\big(2^{-a}\big).
\end{aligned}
\end{equation*}
and for $a>M_2>M_1$, applying \eqref{eq:firstUpper}, \eqref{eq:secondUpper} and \eqref{eq:N2} yields the uniform convergence of $(\gamma_a)_{a\in\N}$.
\end{proof}

\subsection{Other Proofs}
\label{proofs_of_the_preparation_results_push_exact}
In this subsection we complete the rigorous treatment of our main theorems and give the last two remaining proofs. First we prove Lemma \ref{uniform_integrability}, which states that $X_n-\floor{\log_2 n}-\floor{\ln n}$ is uniformly integrable.
\begin{proof}[Proof of Lemma \ref{uniform_integrability}]
Doerr and Künnemann show in \cite[Cor.~3.2 and Thm.~4.1]{doerr_tight_2014} that for all $r\in \mathbb{N}$
\begin{align*}
&P\big[X_n\ge \floor{\log_2 n}+\ln n+2.188+r\big]\le 2e^{-r}\quad\text{and}\\ 
&P\big[X_n\le r\big ]\le P\left [\floor{\log_2 n}-1+\frac{C_n(\ceil{n/2})}{n}\le r\right ],
\end{align*}
 where $C_n(\ceil{n/2})$ is the number of rounds a coupon collector needs to draw the last $n/2$ out of $n$ coupons. These two bounds together with common deviation bounds for the coupon collector problem imply, see e.g.~\cite{erdos1961classical}, that
$$P[Y_n\notin \floor{\log_2 n}+\floor{\ln n} \pm (r+5)]\le 4e^{-r}.$$
Using this inequality we get that for any $N\in \mathbb{N}$
$$\mathbb{E}\Big[|Y_n|^k~\Big|~\mathbbm{1}\big[|Y_n|^k>N\big]\Big]\le \sum _{t\ge \sqrt[k]{N}}(t+5)^k4e^{-t},$$
which implies the claim.
\end{proof}
We close the section with the proof  of Lemmas~\ref{u2_l2_close_enough} and \ref{finishGumbel}. 
\begin{proof}[Proof of Lemma \ref{u2_l2_close_enough}]
First we observe that the $(1-n^{-1/6})$ error term in the definition of $\ell, u$ is negligible as is factors out as a small additional term. Thus it suffices to consider  $\ell=f^{(b)}(L)$ and $u=f^{(b)}(U)$ where $ L =\left (1-2^{-a+2}\right )2^{-a-x}$ and $U=2^{-a-x}$ for some $x\in(0,1]$. We assume that~$a \ge 3$.

We start by showing an analogue to Corollary \ref{fAdditive} but concerning $g$. For all $r\ge s\in [0,1]$, using $1-x\le e^{-x}$,
\begin{align*}
g(r-s)=(r-s)e^{r-s-1}\ge re^{r-s-1}-se^{r-1}\ge g(r)-s(1+r)e^{r-1}
\end{align*}
and consequently
\begin{align}\label{eq:gExp}
g^{(i)}(r-s)\ge g^{(i)}(r)-s\big((1+r)e^{r-1}\big)^i\quad \text{for all } r\ge s\in [0,1) \text{ and } i\in\N.
\end{align}
This completes our preparations. In order to show that $(1-\ell)/(1-u)\to 1$ as $a\to\infty$ we argue that $\ell$ and $u$ are very close together and approach 1 as $a$ (and $b > 2a$) gets big. We start by bounding the distance between $\ell$ and $u$. Applying Corollary \ref{fAdditive} to $U=L+2^{-2a-x+2}$ we get that
\begin{align}\label{fAdditiveAppl}
f^{(a)}\left (U\right )=f^{(a)}\left (L+2^{-2a-x+2}\right )\le f^{(a)}\left (L\right )+2^{-a-x+2}
\end{align}
and Corollary \ref{boundf}  bounds $f^{(a-1)}(U)$ from below with $2^{-x-1}\left (1-2^{-x-1}-2^{-2x-2}\right )\ge 1/8$, thus $f^{(a+2)}(U)\ge 1/2$, and therefore we get using the monotonicity of~$f$
\begin{align}\label{eq:boundF}
\frac{1}{2}\le f^{(a+2)}(U) \le f^{(a+3)}\left (L\right )\le f^{(a+3)}\left (U\right )\overset{\eqref{fAdditiveAppl}}{\le} f^{(a+3)}\left (L\right )+2^{-a-x+5}.
\end{align}
We switch our focus to $g$. Observe that $z:=3e^{-1/2}/2<1$ and, using \eqref{eq:gExp}, 
\begin{align*}
g^{(b-a-3)}\left (1-f^{(a+3)}\left (L\right )-2^{-a-x+5}\right )\ge g^{(b-a-3)}\left (1-f^{(a+3)}\left (L\right )\right )-2^{-a-x+5}\cdot z^{b-a-3}.
\end{align*}
This implies, using \eqref{eq:boundF} and the previous equation, that
\begin{align*}
g^{(b-a-3)}\left (1-f^{(a+3)}\left (U\right )\right)
&\ge g^{(b-a-3)}\left (1-f^{(a+3)}\left (L\right )\right )-2^{-a-x+5}\cdot z^{b-a-3},
\end{align*}
and therefore, as
$1-f^{(b)}(L)\ge1- f^{(b)}(U)=g^{(b-a-3)}\left (1-f^{(a+3)}\left (U\right )\right )$,
\begin{align}\label{eq:diff0}
|u-\ell| = |f^{(b)}(U) - f^{(b)}(L)| \le 2^{-a-x+5}z^{b-a-3}\to 0 \quad \text{as } a\to \infty, b-a\to \infty.
\end{align} 
Next we show that $u,\ell$ approach 1. Using $g(x)=1-f(1-x)$, \eqref{eq:boundF}, $g$ being increasing and $g(x)\le xe^{-1/2}$ for all $x\le 1/2$ (in that order), we get
\begin{align*}
g^{(b)}(1-L)=g^{(b-a-3)}\left (1-f^{(a+3)}\left (L\right )\right )\le g^{(b-a-3)}\left (\frac{1}{2}\right ) \le \frac{1}{2}e^{-(b-a-3)/2}, \quad b > a+3.
\end{align*}
Moreover, using that $f(x)\le 2x$ and $g(x)\ge x/e$,
\begin{align*}
g^{(b)}(1-U)=g^{(b-a)}\left (1-f^{(a)}\left (U\right )\right )\ge \big(1-2^{-x}\big)e^{-(b-a)}.
\end{align*}
Thus, these two bounds together give
\begin{align}\label{eq:diff1}
1- \frac{1}{2}e^{-(b-a-3)/2}\le f^{(b)}\left (L\right )\le f^{(b)}\left (U\right )\le 1-\big(1-2^{-x}\big)e^{-(b-a)}\quad\text{for all } b>a + 3.
\end{align}
We just showed that $u,\ell\to 1$  as $a$ (and $b$) tends to infinity. This yields that $\ln u,\ln \ell$ and $\ln \big(\ell/(e\ell-e+1)\big)$ tend to 0, leaving us with the term $\ln\big((1-\ell)/(1-u)\big)$.
The fact $U\le f(L)$ (and so $f^{(b-2)}(U)\le f^{(b-1)}(L)$) implies that 
\begin{align*}
\frac{1-\ell}{1-u}&= \frac{g^{(b)}\big(1-L\big)}{g^{(b)}\big(1-U\big)}= \frac{\exp\big({g^{(b-1)}(1-L)-1}\big)\cdot \exp\big({g^{(b-2)}(1-L)-1}\big)}{\exp\big({g^{(b-1)}(1-U)-1}\big)\cdot \exp\big({g^{(b-2)}(1-U)-1}\big)}\cdot \frac{g^{(b-2)}\big(1-L\big)}{g^{(b-2)}\big(1-U\big)}\\
&\le \frac{\exp\big({g^{(b-2)}(1-L)-1}\big)}{\exp\big({g^{(b-1)}(1-U)-1}\big)}\cdot \frac{g^{(b-2)}\big(1-L\big)}{g^{(b-2)}\big(1-U\big)}.
\end{align*}
Applying the same estimate to the latter fraction inductively we get for any $c \in\N$
\begin{align*}
\frac{1-\ell}{1-u}\le \frac{\exp\big({g^{(c)}(1-L)-1}\big)}{\exp\big({g^{(b-1)}(1-U)-1}\big)}\cdot \frac{g^{(c)}\big(1-L\big)}{g^{(c)}\big(1-U\big)}\le {\exp\big({g^{(c)}(1-L)}\big)}\cdot \frac{g^{(c)}\big(1-L\big)}{g^{(c)}\big(1-U\big)}.
\end{align*}
Set $c=\ceil{a(1+\ln2)}$. Using \eqref{eq:diff0} and \eqref{eq:diff1}, where we set $b=c$, we obtain (for large enough $a$) that $\big|g^{(c)}(1-L) - g^{(c)}(1-U)\big| \le 2^{-a-x+5}z^{c-a-3}$ as well as $f^{(c)}\big(U\big)\le 1-\big(1-2^{-x}\big)e^{-(c-a)}$ and $f^{(c)}(L)\ge 1-e^{-(c-a-3)/2}/2 $. Thus
\begin{align*}
\frac{1-\ell}{1-u}&\le \exp\left (1-f^{(c)}(L)\right )\left (1+\frac{2^{-a-x+5}z^{c-a-3}}{1-f^{(c)}\big(U\big)}\right )\\
&\le \exp\left (\frac{1}{2}e^{-(c-a-3)/2}\right )\left (1+\frac{2^{-a-x+5}z^{c-a-3}}{(1-2^{-x})e^{-c+a}}\right ).
\end{align*}
Using $e^x\le 1+2x, x\in [0,1]$ this yields the bounds 
\begin{align*}
1\le \frac{1-\ell}{1-u} \le \left (1+\sqrt{2}^{-a+4}\right )\left (1+\frac{2^{-x+5}e}{1-2^{-x}}\cdot z^{a\ln 2 -2}\right )\quad\text{for all }a\in\N.
\end{align*}
Therefore, as $0 < z < 1$ we obtain $\frac{1-\ell}{1-u}\to 1$, and consequently  $\ln\big((1-\ell)/(1-u)\big)\to 0,$ as $a\to \infty.$

\end{proof}
Lemma \ref{finishGumbel} states that disturbing a Gumbel distributed random variable by a small amount does not significantly alter its distribution.
\begin{proof}[Proof of Lemma \ref{finishGumbel}]
Observe that $\big\lceil h(n) + G \pm \varepsilon\big\rceil\neq \big\lceil h(n) + G \big\rceil$
is equivalent to
\begin{align*}
 G\in \big[j-h(n)-\varepsilon, ~j-h(n) +\varepsilon\big]\quad\text{for some }j\ge 1.
\end{align*}
But as $G$ is absolutely continuous, for any $\delta>0$ we can choose $\varepsilon$ small enough such that 
\begin{align*}
P\left [G\in \bigcup_{j\ge 1}\big[j-h(n)-\varepsilon, j-h(n) +\varepsilon\big]\right ]\le \delta.
\end{align*}
\end{proof}

\phantomsection
{
\bibliographystyle{abbrv}
\bibliography{literatur}
}
\appendix
\section{Existence of Subsequence}
Let $x,y\in [0,1]$. In this section we show that there is an unbounded sequence of natural numbers $(n_i)_{i\in\N}$ such that $\log_2 n_i-\floor{\log_2 n_i}\to x$ and $\ln n_i -\floor{\ln n_i}\to y$ as $i\to \infty$.
To this end, set $z=y-x\ln 2$. According to a Theorem of Kronecker, see e.g. \cite[Thm.~440]{hardy1979introduction}, for all $i\in \mathbb{N}$, there are $p_{i},q_{i}\in\mathbb{N}$ such that
\begin{align}\label{eq:Kroecker}
\big|q_{i}\ln 2-p_{i}-z\big|\le i^{-1}.
\end{align}
Actually even more is true: there are infinitely many $p_{i},q_{i} \in \mathbb{N}$ that solve \eqref{eq:Kroecker}. To see this, assume that there are only finitely many, then there is $k,\ell\in \mathbb{N}$ such that $k\ln 2=\ell+z$, otherwise there would be some $i\in \N$ where \eqref{eq:Kroecker} has no solution. However, according to a Theorem of Hurwitz, see e.g.~\cite[Thm.~193]{hardy1979introduction}, there are infinitely many $r_{j},s_{j}\in\mathbb{N}$ such that 
$$\big|r_{j}\ln 2-s_{j}\big|\le r_{j}^{-2}.$$
But then
$$\big|r_{j}\ln 2-s_{j}\big|=\big|(r_{j}+k)\ln 2-(s_{j}+\ell)-z\big|\le r_{j}^{-2},$$
a contradiction, thus there are infinitely many solutions to \eqref{eq:Kroecker}. 
We continue with that equation, which we can restate, as $i\to\infty$,
\begin{align*}
q_{i}\ln 2+x\ln 2=p_{i}+y+O\left ({i}^{-1}\right ).
\end{align*}
Taking the exponential on both sides thus yields, as ${i}\to\infty$, 
\begin{align*}
2^{q_{i}+x}=e^{p_{i}+y+O({i}^{-1})}.
\end{align*}
Set $n_{i}=\floor{2^{q_{i}+x}}$ for all $i\in \N$, where we choose $q_i$ such that  $q_{i}\ge i$ from the infinitely many solutions to \eqref{eq:Kroecker}. Then $n_i\in\N$ for all $i\in \N$ and \begin{align*}
\log_2 n_{i}-\floor{\log_2 n_{i}}=x+O\big(2^{-i}\big)\quad \text{ as well as}\quad \ln n_{i}-\floor{\ln n_{i}}=y+O\left ({i}^{-1}\right ).
\end{align*}
Thus the subsequence of natural numbers that is induced by $\log_2 n_{i}-\floor{\log_2 n_{i}}\to x$ and $\ln n_{i}-\floor{\ln n_{i}}\to y$ is non-empty and unbounded. 
 
\end{document}